\documentclass[11pt,a4paper]{article}
\usepackage{titlesec}
\usepackage{fancyhdr}
\usepackage{a4wide}
\usepackage{graphicx}
\usepackage{float}
\usepackage{amssymb}
\usepackage{amsmath}
\usepackage{amsthm}
\usepackage{color}
\usepackage{mathrsfs}
\usepackage{array}
\usepackage{multido}
\usepackage{wrapfig}
\usepackage[T1]{fontenc}
\usepackage{inputenc}
\usepackage[english]{babel}
\usepackage{tikz}
\usepackage{pgf}
\usetikzlibrary{arrows}
\usepackage{lmodern}
\usepackage{hyperref}
\usepackage{geometry}
\usepackage{changepage}
\changepage{0pt}{}{}{}{}{0pt}{}{0pt}{5pt}
\usepackage[numbers]{natbib}
\usepackage{hyperref}
\hypersetup{
pdfpagemode=UseThumbs,
pdftoolbar=true,        
pdfmenubar=true,        
pdffitwindow=false,     
pdfstartview={Fit},    
pdftitle={Non-scattering wavenumbers and far field invisibility for a finite set of incident/scattering directions},    
pdfauthor={A.-S. Bonnet-Ben Dhia, L. Chesnel, S.A. Nazarov},     
pdfcreator={A.-S. Bonnet-Ben Dhia, L. Chesnel, S.A. Nazarov},   
pdfproducer={A.-S. Bonnet-Ben Dhia, L. Chesnel, S.A. Nazarov}, 
pdfkeywords={}, 
pdfnewwindow=true,      
colorlinks=true,       
linkcolor=magenta,          
citecolor=red,        
filecolor=cyan,      
urlcolor=blue           
}

\newcommand{\dsp}{\displaystyle}

\newcommand{\eps}{\varepsilon}

\newcommand{\Om}{\Omega}
\newcommand{\mrm}[1]{\mathrm{#1}}

\newcommand{\bfx}{\boldsymbol{x}}

\newcommand{\bfnu}{\boldsymbol{\nu}}

\newcommand{\Cplx}{\mathbb{C}}

\newcommand{\R}{\mathbb{R}}

\renewcommand{\div}{\mrm{div}}
\newcommand{\mL}{\mrm{L}}
\newcommand{\mH}{\mrm{H}}

\newcommand{\mA}{\mathscr{A}}
\newcommand{\mB}{\mathscr{B}}

\newcommand{\loc}{\mbox{\scriptsize loc}}

\renewcommand{\ker}{\mrm{ker}}

\newtheorem{lemma}{Lemma}[section]
\newtheorem{remark}{Remark}[section]

\newtheorem{proposition}{Proposition}[section]

\newtheorem{Assumption}{Assumption}[section]

\begin{document}

~\vspace{0.0cm}
\begin{center}
{\sc \bf\LARGE 
Non-scattering wavenumbers and far field invisibility \\[6pt] for a finite set of incident/scattering directions}
\end{center}

\begin{center}
\textsc{Anne-Sophie Bonnet-Ben Dhia}$^1$, \textsc{Lucas Chesnel}$^2$, \textsc{Sergei A. Nazarov}$^{3,\,4,\,5}$\\[16pt]
\begin{minipage}{0.9\textwidth}
{\small
$^1$ Laboratoire Poems, UMR 7231 CNRS/ENSTA/INRIA, Ensta ParisTech, 828, Boulevard des Maréchaux, 91762 Palaiseau, France; \\
$^2$ Centre de mathématiques appliquées, \'{E}cole Polytechnique, 91128 Palaiseau, France; \\
$^3$ Faculty of Mathematics and Mechanics, St. Petersburg State University, Universitetsky prospekt, 28, 198504, Peterhof, St. Petersburg, Russia;\\
$^4$ Laboratory for mechanics of new nanomaterials, St. Petersburg State Polytechnical University, Polytekhnicheskaya ul, 29, 195251, St. Petersburg, Russia;\\
$^5$ Laboratory of mathematical methods in mechanics of materials, Institute of Problems of Mechanical Engineering, Bolshoj prospekt, 61, 199178, V.O., St. Petersburg, Russia.\\[10pt] 
E-mail: \texttt{Anne-Sophie.Bonnet-Bendhia@ensta-paristech.fr}, \texttt{lucas.chesnel@cmap.polytechnique.fr}, \texttt{srgnazarov@yahoo.co.uk}\\[-14pt]
\begin{center}
(\today)
\end{center}
}
\end{minipage}
\end{center}
\vspace{0.4cm}

\noindent\textbf{Abstract.} 
We investigate a time harmonic acoustic scattering problem by a penetrable inclusion with compact support embedded in the free space. We consider cases where an observer can produce incident plane waves and measure the far field pattern of the resulting scattered field only in a finite set of directions. In this context, we say that a wavenumber is a non-scattering wavenumber if the associated relative scattering matrix has a non trivial kernel. Under certain assumptions on the physical coefficients of the inclusion, we show that the non-scattering wavenumbers form a (possibly empty) discrete set. Then, in a second step, for a given real wavenumber and a given domain $\mathcal{D}$, we present a constructive technique to prove that there exist inclusions supported in $\overline{\mathcal{D}}$ for which the corresponding relative scattering matrix is null. These inclusions have the important property to be impossible to detect from far field measurements. The approach leads to a numerical algorithm which is described at the end of the paper and which allows to provide examples of (approximated) invisible inclusions.\\


\section{Introduction}

Consider a reference acoustic medium, say $\R^d$, $d=2,3$, presenting a defect (penetrable inclusion) localized in the bounded domain $\mathcal{D}$. Generating incident plane waves and measuring the resulting far field pattern of the scattered fields (the only available information far from $\mathcal{D}$), one can try to reconstruct the features of the defect in the reference medium. In particular, the classical inverse scattering problem is to determine the support of the inclusion. In view of this aim, many methods have been developed \cite{Pott06,CaCo06}, as for example, the \textit{Linear Sampling Method} (LSM) \cite{CoKi96,CoPP97}. It seems that the LSM works, without modifications and \textit{a priori} knowledge, only at wavenumbers which are not \textit{transmission eigenvalues}. Transmission eigenvalues correspond to wavenumbers $k>0$ such that there are incident fields, generalized\footnote{Here, the adjective ``generalized'' means that the incident field can be the combination of an infinite number of plane waves. In the literature, such a field is referred to as a Herglotz wave function.} combination of plane waves, which produce arbitrarily small scattered fields outside the inclusion. This, and the fact that nice questions of spectral theory appear in its study, explains why the \textit{interior transmission eigenvalue problems } has been so intensely investigated for now more than 25 years (see \cite{Kirs86,CoMo88,RySl91} and the recent review paper \cite{CaHa13}).\\
\newline
All the aforementioned theory supposed that we can produce incident plane waves in all directions $\boldsymbol{\theta}_{\mrm{i}}\in \mathbb{S}^{d-1}$ and measure the far field pattern of the associated scattered fields for all $\boldsymbol{\theta}_{\mrm{s}}\in \mathbb{S}^{d-1}$, where $\mathbb{S}^{d-1}$ denotes the unit sphere in $\R^d$. However, this is not quite realistic. Some works have been devoted to the study of the LSM in partial aperture (see \cite{Cako07} and \cite{CoMH03,CoHP03,CoMo06} for the corresponding numerical experiments), that is when the observer knows only the restriction of the far field operator on a non empty open set of $\mathbb{S}^{d-1}$. Thanks to the analytic dependence with respect to the wavenumber, the theory can be developed as in the case of full aperture. But again, this is not completely satisfactory for applications because in practice, one cannot access to this continuum of information. Very often, for example in the numerical implementation of the LSM, one can produce incident plane waves and measure the far field pattern of the associated scattered fields only in a finite set of directions. To study such a configuration, let us consider given $N$ distinct incident directions $\boldsymbol{\theta}_1,\dots, \boldsymbol{\theta}_N\in\mathbb{S}^{d-1}$. We shall assume that the emitters and the receivers are located at the same positions so that measurements can be made only in the directions $-\boldsymbol{\theta}_1,\dots, -\boldsymbol{\theta}_N$ (backscattering directions). Then, the question of the injectivity of the far field operator amounts to the question of the injectivity of a relative scattering matrix of size $N\times N$ denoted $\mA(k)$. In this context, we shall say that $k>0$ is a non-scattering wavenumber if $\mA(k)$ is not injective, equivalently, if there is a non trivial incident field, combination of the plane waves of directions of propagation $\boldsymbol{\theta}_1,\dots,\boldsymbol{\theta}_N$, such that the far field pattern of the associated scattered field vanishes in the directions $-\boldsymbol{\theta}_1,\dots,-\boldsymbol{\theta}_N$. Let us emphasize that in this case, unlike in the continuous setting, the scattered field has no reason to be identically null outside the defect. The first objective of the present article is to find criteria on the physical parameters of the inclusion to guarantee that the non-scattering wavenumbers, defined by means of the relative scattering matrix, form a (possibly empty) discrete set. This kind of results seems an important first step to justify the practical implementation of reconstruction methods such as the LSM mentioned above.\\
\newline
On the other hand, all techniques are not equally sensitive to transmission eigenvalues. Thus, it has been shown in \cite{Lech09} that the \textit{Factorization Method} (FM) \cite{KiGr08} is stable at transmission eigenvalues. Again, this property has been demonstrated in the continuous framework. Can we hope to justify the FM in the setting with a finite set of emitters and receivers for all wavenumbers $k>0$? This is the second question we investigate in this work. More precisely, under certain assumptions on the directions $\boldsymbol{\theta}_1,\dots,\boldsymbol{\theta}_N$, we show that the answer is negative proving constructively that for any $k>0$ and any domain $\mathcal{D}$, there are penetrable inclusions supported in $\overline{\mathcal{D}}$ for which the matrix $\mA(k)$ is null (that is equivalent to say that $\mA(k)$ has a kernel of dimension $N$). Let us underline that we consider only non dissipative isotropic inclusions, the result being simpler to establish for dissipative or anisotropic inclusions. The above proposition indicates in particular that the data of the relative scattering matrix does not uniquely determine the position of the defect. Bearing in mind that $\mA(k)$ belongs to a space of finite dimension while there is an infinite number of degrees of freedom for the definition of the defect, this result is not completely surprising. However, we do not know any existing proof in the literature. Note that in the continuous setting, this construction is impossible. Indeed, the knowledge of the far field operator on $\mathbb{S}^{d-1}\times\mathbb{S}^{d-1}$ uniquely determines the parameter of the inclusion (see \cite[Theorem 10.5]{CoKr13} and the references therein: \cite{Nach88,SyUh87} in 3D, \cite{Bukh08,ImYa12} in 2D).\\
\newline
This text is organized as follows. In Section \ref{sectionSetting}, we formulate the problem and introduce the notations. Section \ref{SectionFreeSpace} is dedicated to the proof of discreteness of non-scattering wavenumbers in the setting with a finite set of emitters and receivers. We first give a sense to the relative scattering matrix for complex wavenumbers. Then, we establish some energy identities for purely imaginary wavenumbers $k$. Using the analytic dependence on $k$, these equalities allow us to prove that the real wavenumbers $k$ for which $\mA(k)$ is not injective form a (possibly empty) discrete set. In the process, we also consider sound hard obstacles for which the analysis is slightly simpler. In Section \ref{FFInvisibility}, we adopt a different point of view. For a given wavenumber $k>0$ and a given domain $\mathcal{D}$, we present a constructive technique to demonstrate that there are non trivial inclusions supported in $\overline{\mathcal{D}}$ for which $\mA(k)$ is the null matrix. To implement the approach, which was developed in \cite{Naza11,BoNa13} in a context related to waveguides problems, we need in particular to assume that the scattering directions are different from the incident directions. In the second part of Section \ref{FFInvisibility}, we show that it is much more complicated (it might be impossible but we are not able to prove it) to impose far field invisibility in the incident direction. Finally in the last part of Section \ref{FFInvisibility}, we derive a numerical algorithm to provide examples of (approximated) invisible inclusions.

\section{Setting}\label{sectionSetting}
We assume that the propagation of acoustic waves in time harmonic regime in the reference medium $\R^d$, $d=2,3$, is governed by the Helmholtz equation $\Delta u+k^2 u=0$. Here, $u$ denotes the pressure, $k>0$ corresponds to the wavenumber proportional to the frequency of harmonic oscillations and $\Delta$ is the Laplace operator. The localized perturbation in the reference medium is modelled by some bounded open set $\mathcal{D}\subset\R^d$ with Lipschitz boundary $\Gamma:=\partial\mathcal{D}$. To represent the physical properties of the penetrable inclusion $\mathcal{D}$, we introduce $A\in\mrm{L}^{\infty}(\R^d,\R^{d\times d})$ and $\rho\in \mrm{L}^{\infty}(\R^d,\R)$, two parameters such that $A=\mrm{Id}$ and $\rho=1$ in $\R^d\setminus\overline{\mathcal{D}}$. We assume that $A(\bfx)$ is symmetric for all $\bfx\in\mathcal{D}$, satisfies $A(\bfx)\zeta\cdot\zeta\ge C |\zeta|^2$ for all $\bfx\in\R^d$, $\zeta\in\R^d$ and that $\rho$ verifies $\rho(\bfx)\ge C$ for all $\bfx\in\R^d$, for some constant $C>0$. The scattering of the incident plane wave $u_{\mrm{i}}:=e^{i k \boldsymbol{\theta}_{\mrm{i}}\cdot\boldsymbol{x}}$, of direction of propagation $\boldsymbol{\theta}_{\mrm{i}}\in\mathbb{S}^{d-1}$, by $\mathcal{D}$, is described by the problem
\begin{equation}\label{PbChampTotalFreeSpaceIntrod}
\begin{array}{|rcll}
\multicolumn{4}{|l}{\mbox{Find }u\mbox{ such that}}\\[4pt]
-\div(A \nabla u) & = & k^2 \rho\,u & \mbox{ in }\R^d,\\
u & = & u_{\mrm{i}}+u_{\mrm{s}} & \mbox{ in }\R^d,\\[4pt]
\multicolumn{4}{|c}{\dsp\lim_{r\to +\infty} r^{\frac{d-1}{2}}\left( \frac{\partial u_{\mrm{s}}}{\partial r}-ik u_{\mrm{s}} \right)  = 0.}
\end{array}
\end{equation}
We emphasize that in (\ref{PbChampTotalFreeSpaceIntrod}), $u_{\mrm{i}}$ is given. On the other hand, the last line of (\ref{PbChampTotalFreeSpaceIntrod}), where $r=|\boldsymbol{x}|$, is the Sommerfeld radiation condition which ensures that the scattered field $u_{\mrm{s}}$ is outgoing. It is known (see for example \cite{CoKr13}) that for all $k>0$, Problem (\ref{PbChampTotalFreeSpaceIntrod}) has a unique solution in $\mH^1_{\loc}(\R^d)$, the Sobolev space of measurable functions whose $\mH^1$-norm is finite on each bounded subset of $\R^d$. We shall denote $u_{\mrm{s}}(\cdot,\boldsymbol{\theta}_{\mrm{i}})$ the associated scattered field. Using Green's representation formula and computing explicitly the asymptotic behaviour of the Green's function for the Helmholtz equation far from $\mathcal{D}$, one proves (see \cite[Theorem 2.6]{CoKr13}) the expansion 
\begin{equation}\label{scatteredFieldFreeSpaceIntro}
u_{\mrm{s}}(\boldsymbol{x},\boldsymbol{\theta}_{\mrm{i}})= \dsp e^{i k r}r^{-\frac{d-1}{2}}\,\Big(\,u_{\mrm{s}}^{\infty}(\boldsymbol{\theta}_{\mrm{s}},\boldsymbol{\theta}_{\mrm{i}})+O(1/r)\,\Big),
\end{equation}
as $r\to+\infty$, uniformly in $\boldsymbol{\theta}_{\mrm{s}}\in \mathbb{S}^{d-1}$. Here $\boldsymbol{\theta}_{\mrm{s}}$ is the direction of observation. The function $u^{\infty}_{\mrm{s}}(\cdot,\boldsymbol{\theta}_{\mrm{i}}): \mathbb{S}^{d-1}\to\Cplx$,  is called the \textit{far field pattern} associated with $u_{\mrm{i}}:=e^{i k \boldsymbol{\theta}_{\mrm{i}}\cdot\boldsymbol{x}}$.
In other words, at infinity, $u_{\mrm{s}}(\cdot,\boldsymbol{\theta}_{\mrm{i}})$ behaves at the first order like a cylindrical wave in 2D or like a spherical wave in 3D. The far field pattern is given by the following integral representation
\begin{equation}\label{FarFieldPatternGene}
u_{\mrm{s}}^{\infty}(\boldsymbol{\theta}_{\mrm{s}},\boldsymbol{\theta}_{\mrm{i}}):= c_d  \langle \partial_{\bfnu}u_{\mrm{s}}^+(\cdot,\boldsymbol{\theta}_{\mrm{i}}),e^{-ik \boldsymbol{\theta}_{\mrm{s}}\cdot\boldsymbol{x}} \rangle_{\Gamma}-c_d  \langle \partial_{\bfnu}( e^{-ik \boldsymbol{\theta}_{\mrm{s}}\cdot\boldsymbol{x}}),u_{\mrm{s}}(\cdot,\boldsymbol{\theta}_{\mrm{i}}) \rangle_{\Gamma}.
\end{equation}
In this expression, $\langle\cdot,\cdot\rangle_{\Gamma}$ stands for the duality pairing (without complex conjugation) between $\mH^{-1/2}(\Gamma)$ and $\mH^{1/2}(\Gamma)$. On the other hand, $\boldsymbol{\nu}$ denotes the unit normal vector to $\Gamma$ orientated to the interior of $\mathcal{D}$ and $\partial_{\bfnu}u_{\mrm{s}}^+(\cdot,\boldsymbol{\theta}_{\mrm{i}})$ refers to the normal trace of $u_{\mrm{s}}(\cdot,\boldsymbol{\theta}_{\mrm{i}})|_{\Om}$, with $\Om:=\R^d\setminus\overline{\mathcal{D}}$.
Finally, the constant $c_d$ verifies $c_2=e^{i\pi/4}/\sqrt{8\pi k}$ and $c_3=1/(4\pi)$. In particular when $A=\mrm{Id}$ in $\R^d$, we deduce from (\ref{FarFieldPatternGene}) that 
\begin{equation}\label{FFPtermeSource}
u_{\mrm{s}}^{\infty}(\boldsymbol{\theta}_{\mrm{s}},\boldsymbol{\theta}_{\mrm{i}}) = c_d\,k^2\dsp \int_{\mathcal{D}} (\rho-1)(u_{\mrm{s}}^{\infty}(\cdot,\boldsymbol{\theta}_{\mrm{i}})+e^{ik \boldsymbol{\theta}_{\mrm{i}}\cdot\bfx}) \,e^{-ik \boldsymbol{\theta}_{\mrm{s}}\cdot\bfx}\,d\bfx.
\end{equation}
As mentioned in the introduction, we shall assume in this article that we have a finite set of emitters and receivers located at the same positions so that we can produce incident plane waves in some given directions $\boldsymbol{\theta}_1,\dots, \boldsymbol{\theta}_N\in\mathbb{S}^{d-1}$ and measure the far field pattern of the resulting scattered field only in the  directions $-\boldsymbol{\theta}_1,\dots, -\boldsymbol{\theta}_N$. This corresponds to knowing all elements of the relative scattering matrix $\mA(k)\in\Cplx^{N\times N}$ such that
\begin{equation}\label{defRelScaMat}
\mA_{mn}(k) = u_{\mrm{s}}^{\infty}(-\boldsymbol{\theta}_m,\boldsymbol{\theta}_n).
\end{equation}
In the next section, we prove that for given physical parameters $A$ and  $\rho$, the values of $k>0$ such that $\mA(k)$ is not injective form a (possibly empty) discrete set. Then in Section \ref{FFInvisibility}, imposing $A=\mrm{Id}$ in $\R^d$, for a given $k>0$ and a given domain $\mathcal{D}$, we construct a $\rho$ supported in $\overline{\mathcal{D}}$ such that $\mA(k)$ is the null matrix. 

\section{Discreteness of non-scattering wavenumbers}\label{SectionFreeSpace}

We remind the reader that the wavenumber $k>0$ is called non-scattering wavenumber if $\mA(k)$ is not injective. Before working on the relative scattering matrix $\mA(k)$, we investigate first the case where $\mathcal{D}$ is a sound hard obstacle (rather than a penetrable inclusion). This study is convenient for pedagogical purposes because the analysis for the sound hard obstacle is (slightly) simpler than the one for the penetrable inclusion. On the other hand, it yields a result which is interesting by itself. Indeed, while the discreteness of transmission eigenvalues for the sound hard obstacle in the continuous setting is obtained for free (they correspond to the eigenvalues for the Neumann Laplacian in $\mathcal{D}$), the equivalent result is more delicate to show in the discrete framework (note that transmission eigenvalues are then called non-scattering wavenumbers).

\subsection{Sound hard obstacle} 
Denote $\Om=\R^d\setminus\overline{\mathcal{D}}$ and consider the following scattering problem 
\begin{equation}\label{PbChampTotalFreeSpace}
\begin{array}{|rcll}
\multicolumn{4}{|l}{\mbox{Find }u\mbox{ such that}}\\
-\Delta u  & = & k^2 u & \mbox{ in }\Om\\
u & = & u_{\mrm{i}}+u_{\mrm{s}} & \mbox{ in }\Om\\
\partial_{\bfnu}u  & = & 0  & \mbox{ on }\Gamma\\[4pt]
\multicolumn{4}{|c}{\dsp\lim_{r\to +\infty} r^{\frac{d-1}{2}}\left( \frac{\partial u_{\mrm{s}}}{\partial r}-ik u_{\mrm{s}} \right)  = 0,}
\end{array}
\end{equation}
with $u_{\mrm{i}}=e^{i k \boldsymbol{\theta}_{\mrm{i}}\cdot\boldsymbol{x}}$. In (\ref{PbChampTotalFreeSpace}), the third equation on the boundary $\Gamma$ models the sound hard obstacle. According e.g. to \cite[Theorem 9.11]{McLe00}, we know that Problem (\ref{PbChampTotalFreeSpace}) has a unique solution in $\mrm{H}^1_{\loc}(\R^d)$ for all $k>0$. We call $u(\cdot,\boldsymbol{\theta}_{\mrm{i}})$ and $u_{\mrm{s}}(\cdot,\boldsymbol{\theta}_{\mrm{i}})$ the corresponding total and scattered fields. Formula (\ref{FarFieldPatternGene}), which is also valid for this problem, and a simple integration by parts on $\mathcal{D}$ provide the equalities
\begin{equation}\label{equationsFFP}
\begin{array}{ccl}
u_{\mrm{s}}^{\infty}(\boldsymbol{\theta}_{\mrm{s}},\boldsymbol{\theta}_{\mrm{i}}) & = & c_d  \langle \partial_{\bfnu}u_{\mrm{s}}(\cdot,\boldsymbol{\theta}_{\mrm{i}}),e^{-ik \boldsymbol{\theta}_{\mrm{s}}\cdot\boldsymbol{x}} \rangle_{\Gamma}-c_d  \langle \partial_{\bfnu}( e^{-ik \boldsymbol{\theta}_{\mrm{s}}\cdot\boldsymbol{x}}),u_{\mrm{s}}(\cdot,\boldsymbol{\theta}_{\mrm{i}}) \rangle_{\Gamma}\\[5pt]
0 & = &  c_d  \langle \partial_{\bfnu}u_{\mrm{i}},e^{-ik \boldsymbol{\theta}_{\mrm{s}}\cdot\boldsymbol{x}} \rangle_{\Gamma}-c_d  \langle \partial_{\bfnu}( e^{-ik \boldsymbol{\theta}_{\mrm{s}}\cdot\boldsymbol{x}}),u_{\mrm{i}} \rangle_{\Gamma}.
\end{array}
\end{equation} 
Summing the two equations of (\ref{equationsFFP}) and remembering that $\partial_{\bfnu}(u_{\mrm{i}}+u_{\mrm{s}}(\cdot,\boldsymbol{\theta}_{\mrm{i}}))=\partial_{\bfnu}u(\cdot,\boldsymbol{\theta}_{\mrm{i}})=0$ on $\Gamma$, we deduce that 
\begin{equation}\label{FFPFreeSpace}
u_{\mrm{s}}^{\infty}(\boldsymbol{\theta}_{\mrm{s}},\boldsymbol{\theta}_{\mrm{i}})= -c_d  \langle \partial_{\bfnu}( e^{-ik \boldsymbol{\theta}_{\mrm{s}}\cdot\boldsymbol{x}}),u(\cdot,\boldsymbol{\theta}_{\mrm{i}})\rangle_{\Gamma}.
\end{equation}
We denote $\mB(k)\in\Cplx^{N\times N}$ the relative scattering matrix for this problem. It is defined elementwise via
\begin{equation}\label{defScaMatB}
\mB_{mn}(k)=u_{\mrm{s}}^{\infty}(-\boldsymbol{\theta}_m,\boldsymbol{\theta}_n)=-c_d  \langle \partial_{\bfnu}( e^{ik \boldsymbol{\theta}_m\cdot\boldsymbol{x}}),u(\cdot,\boldsymbol{\theta}_n)\rangle_{\Gamma}.
\end{equation}
As for the case of the penetrable inclusion, we shall say that $k\in\R^{\ast}$ is a non-scattering wavenumber when the kernel of $\mB(k)$ contains a non-zero element. In the sequel, we wish to prove that the set of non-scattering wavenumbers is either empty or discrete. \\
\newline
We start by giving a sense to $\mB(k)$ for non real $k$. For $k\in\Cplx\setminus\R$, $f\in\mL^2_{c}(\Om)$ (the set of functions of $\mL^2(\Om)$ which are compactly supported) and $g\in\mH^{-1/2}(\Gamma)$, the Lax-Milgram theorem ensures there is a unique solution $u\in\mH^1(\Om)$ to the following problem
\begin{equation}\label{PbGeneFreeSpace}
\begin{array}{|rcll}
\multicolumn{4}{|l}{\mbox{Find }u\mbox{ such that}}\\
-\Delta u  - k^2 u& = & f & \mbox{ in }\Om\\
\partial_{\bfnu}u  & = & g  & \mbox{ on }\Gamma.
\end{array}
\end{equation}
In particular, this allows us to define, for $k$ verifying $\Im m\,k>0$, the resolvent $\mathcal{R}(k):\mL^2_{c}(\Om)\times\mH^{-1/2}(\Gamma)\to \mH^1_{\loc}(\Om)$ such that $\mathcal{R}(k)(f,g)=u$. The map $\mathcal{R}$ can be meromorphically continued to $\Lambda$ where $\Lambda$ is equal to $\Cplx$ or to some Riemann surface according to the parity of the space dimension (the interested reader can find more details for example in \cite{Melr95}). Moreover, $\mathcal{R}$ has no pole on $(0;+\infty)$ and, according to the limiting absorption principle, for $k\in (0;+\infty)$, we have  $\mathcal{R}(k)(0,-\partial_{\bfnu}u_{\mrm{i}})=u_{\mrm{s}}$ where $u_{\mrm{s}}$ is defined in (\ref{PbChampTotalFreeSpace}). We deduce that the matrix valued map $k\mapsto \mB(k)$ introduced in (\ref{defScaMatB}) is analytic on $\Lambda \setminus \mathscr{P}$, $\mathscr{P}$ denoting the set of poles of $\mathcal{R}$. As a consequence, $k\mapsto \det\mB(k)$ is analytic on $\Lambda \setminus \mathscr{P}$ and, according to the principle of isolated zeros, to show that the set of non-scattering wavenumbers is either empty or discrete, it is sufficient to exhibit some $k\in\Lambda \setminus \mathscr{P}$ such that $\mB(k)$ is injective. We will prove that this is true when $k$ is purely imaginary with $\Im m\,k>0$ (note that $\{k\in\Cplx\,|\,\Im m\,k>0 \}$ is indeed included in $\Lambda \setminus \mathscr{P}$).\\
\newline 
Set $k=i\kappa$ with $\kappa>0$. In this case, if the source terms in (\ref{PbGeneFreeSpace}) are real valued, then the solution in $\mH^1(\Om)$ of Problem (\ref{PbGeneFreeSpace}) is also real-valued. Since the functions $\boldsymbol{x}\mapsto e^{ik\boldsymbol{\theta}_n\cdot\boldsymbol{x}}$ are real valued, it is sufficient to prove that $\mB(i\kappa)$ is invertible as a matrix of $\R^{N\times N}$. Consider the incident field
\begin{equation}\label{incidentFieldBS}
u_{\mrm{i}}=\sum_{n=1}^N \alpha_n e^{ik\boldsymbol{\theta}_n\cdot\boldsymbol{x}}=\sum_{n=1}^N \alpha_n e^{-\kappa \boldsymbol{\theta}_n\cdot\boldsymbol{x}},\qquad (\alpha_1,\dots,\alpha_N)^{\top}\in\R^{N}.
\end{equation}
We call $u$ the unique solution to the problem
\begin{equation}\label{PbKcomp}
\begin{array}{|rcll}
\multicolumn{4}{|l}{\mbox{Find }u\mbox{ such that}}\\
-\Delta u  & = & -\kappa^2 u & \mbox{ in }\Om\\
\partial_{\bfnu}u  & = & 0  & \mbox{ on }\Gamma \\
u_{\mrm{s}} & := & u-u_{\mrm{i}} & \in\mH^1(\Om).
\end{array}
\end{equation}
If $u(\cdot,\boldsymbol{\theta}_n)$ denotes the solution of (\ref{PbKcomp}) with $u_{\mrm{i}}=e^{-\kappa\boldsymbol{\theta}_n\cdot\boldsymbol{x}}$, then by linearity, we have $u=\sum_{n=1}^N \alpha_nu(\cdot,\boldsymbol{\theta}_n)$. Since $u=u_{\mrm{s}}+u_{\mrm{i}}$ in $\Om$ and $\partial_{\bfnu}u=0$ on $\Gamma$,  we have $u_{\mrm{s}}=u-u_{\mrm{i}}$ and  $\partial_{\bfnu}u_{\mrm{s}}=-\partial_{\bfnu}u_{\mrm{i}}$ on $\Gamma$. This allows us to write
\begin{equation}\label{estimEner1FreeSpace}
\begin{array}{lcl}
\dsp \int_{\Om} |\nabla u_{\mrm{s}}|^2+\kappa^2u_{\mrm{s}}^2\,d\bfx & = & \dsp \int_{\Om} |\nabla u_{\mrm{s}}|^2+ u_{\mrm{s}}\,\Delta u_{\mrm{s}}\,d\bfx \\[10pt]
 & = & \langle \partial_{\bfnu} u_{\mrm{s}},u_{\mrm{s}}  \rangle_{\Gamma}\ =\ -\langle \partial_{\bfnu} u_{\mrm{i}},(u-u_{\mrm{i}})  \rangle_{\Gamma}\\[10pt]
& = & -\dsp \int_{\mathcal{D}} |\nabla u_{\mrm{i}}|^2+\kappa^2u_{\mrm{i}}^2\,d\bfx -\langle \partial_{\bfnu} u_{\mrm{i}},u  \rangle_{\Gamma}.
\end{array}
\end{equation}
The ``-'' in front of the first term of the right-hand side of the above equation appears because $\boldsymbol{\nu}$ is orientated to the interior of $\mathcal{D}$. On the other hand, 
using (\ref{defScaMatB}), which also holds for $k$ such that $\Im m\,k>0$\footnote{Throughout the paper, the complex square root is chosen so that if $\xi=r e^{i\gamma}$ for $r\ge 0$ and $ \gamma \in[0;2\pi)$, then $\sqrt{\xi}=\sqrt{r}e^{i\gamma/2}$. With this choice, there holds $\Im m\,\sqrt{\xi}\ge 0$ for all $\xi\in\Cplx$.}, we find
\begin{equation}\label{estimEner2FreeSpace}
\langle \partial_{\bfnu} u_{\mrm{i}},u  \rangle_{\Gamma} = \sum_{m=1}^N\sum_{n=1}^N \alpha_m\alpha_n \langle \partial_{\bfnu} e^{-\kappa\boldsymbol{\theta}_m\cdot\boldsymbol{x}},u(\cdot,\boldsymbol{\theta}_n)\rangle_{\Gamma}= -c_d^{-1}\,\alpha^{\top}\mB(i\kappa)\,\alpha,
\end{equation}
where $\alpha=(\alpha_1,\dots,\alpha_N)^{\top}$. Gathering (\ref{estimEner1FreeSpace}) and (\ref{estimEner2FreeSpace}), we obtain the energy identity
\begin{equation}\label{estimEnerTotalFreespace1}
c_d^{-1}\,\alpha^{\top}\mB(i\kappa)\,\alpha  = \dsp \int_{\Om} |\nabla u_{\mrm{s}}|^2+\kappa^2u_{\mrm{s}}^2\,d\bfx +\dsp \int_{\mathcal{D}} |\nabla u_{\mrm{i}}|^2+\kappa^2u_{\mrm{i}}^2\,d\bfx.
\end{equation}
Assume that $\mathcal{D}\ne\emptyset$. If $\alpha=(\alpha_1,\dots,\alpha_N)^{\top}$ is an element of $\ker\,\mB(i\kappa)$, then, according to (\ref{estimEnerTotalFreespace1}), the field $u_{\mrm{i}}=\sum_{n=1}^N\alpha_n e^{-\kappa\boldsymbol{\theta}_n\cdot\boldsymbol{x}}$
must satisfy $u_{\mrm{i}}=0$ on $\mathcal{D}$. This implies that $\alpha$ is the null vector and proves that $\mB(i\kappa)$ is injective (or equivalently $\det\mB(i\kappa)\ne0$). Since the zeros of the analytic function $k\mapsto \det\mB(k)$ are isolated, we deduce the following result:
\begin{proposition}\label{PropoSoundSoftFS}\textsc{(Sound hard obstacle) --}
The set of non-scattering wavenumbers for Problem (\ref{PbChampTotalFreeSpace}) is either empty or discrete.
\end{proposition}
\begin{remark}
The case of the sound soft obstacle, for which we replace, in (\ref{PbChampTotalFreeSpace}), the homogeneous Neumann boundary condition by a homogeneous Dirichlet boundary condition, can be treated in a similar way.
\end{remark}

\subsection{Penetrable inclusion}
\noindent We come back to the problem of the penetrable inclusion (Problem (\ref{PbChampTotalFreeSpaceIntrod})). As for $\mB(k)$, the map $k\mapsto \mA(k)$, where $\mA(k)\in\Cplx^{N\times N}$ is the relative scattering matrix defined in (\ref{defRelScaMat}), can be meromorphically continued to $\Lambda$, with $\Lambda$ equal to $\Cplx$ or to some Riemann surface according to the parity of the space dimension. Therefore, again, to prove that the set of non-scattering wavenumbers is either empty or discrete, it is sufficient to exhibit some $k=i\kappa$, with $\kappa>0$, such that $\mA(k)$ is injective. As explained above, it is sufficient to prove that $\mA(i\kappa)$ is invertible as an element of $\R^{N\times N}$. Consider an incident field like in (\ref{incidentFieldBS}) (we use the same notation). We call $u$ the unique solution to the problem
\[
\begin{array}{|rcll}
\multicolumn{4}{|l}{\mbox{Find }u\mbox{ such that}}\\
-\div(A\nabla u)  & = & k^2\rho\,u & \mbox{ in }\R^d\\
u_{\mrm{s}} & := & u-u_{\mrm{i}} & \in\mH^1(\Om).
\end{array}
\]
As in the previous subsection, we wish to establish energy identities to show that $\mA(i\kappa)\in\R^{N\times N}$ is invertible. Using (\ref{equationsFFP}) and working as in (\ref{estimEner2FreeSpace}), we derive
\begin{equation}\label{estimEnerTotal1}
\begin{array}{lcl}
c_d^{-1}\,\alpha^{\top}\mA(i\kappa)\,\alpha & = & \langle \partial_{\bfnu}u^{+}_{\mrm{s}},u_{\mrm{i}} \rangle_{\Gamma}-\langle \partial_{\bfnu}u^{+}_{\mrm{i}},u_{\mrm{s}}  \rangle_{\Gamma}
\end{array}
\end{equation}
In the right hand side of (\ref{estimEnerTotal1}), $\partial_{\bfnu}u^{+}_{\mrm{s}}$ and $\partial_{\bfnu} u^{+}_{\mrm{i}}$ refer, respectively, to the normal trace of $u_{\mrm{s}}|_{\Om}$ and $u_{\mrm{i}}|_{\Om}$ on $\Gamma$. We shall denote $\partial_{\bfnu} u^{-}_{\mrm{i}}$ the normal trace of $u_{\mrm{i}}|_{\mathcal{D}}$ on $\Gamma$. We define similarly $A\partial_{\bfnu} u^{-}$ and $\partial_{\bfnu} u^{+}$. Notice that on $\Gamma$, we have $A\partial_{\bfnu} u^{-}=\partial_{\bfnu} u^{+}$, $\partial_{\bfnu} u^{-}_{\mrm{i}}=\partial_{\bfnu} u^{+}_{\mrm{i}}$ and $\partial_{\bfnu} u^{+}=\partial_{\bfnu} u^{+}_{\mrm{i}}+\partial_{\bfnu} u^{+}_{\mrm{s}}$. Since $u_{\mrm{i}}=u-u_{\mrm{s}}$, we obtain
\begin{equation}\label{estimEnerTotal2}
\begin{array}{lcl}
\langle \partial_{\bfnu}u^{+}_{\mrm{s}},u_{\mrm{i}} \rangle_{\Gamma}-\langle \partial_{\bfnu}u^{+}_{\mrm{i}},u_{\mrm{s}}  \rangle_{\Gamma} & = &  \langle \partial_{\bfnu}u^{+},u_{\mrm{i}}  \rangle_{\Gamma}-\langle \partial_{\bfnu}u^{+}_{\mrm{i}},u_{\mrm{i}}  \rangle_{\Gamma} + \langle \partial_{\bfnu}u^{+}_{\mrm{s}},u_{\mrm{s}}  \rangle_{\Gamma}-\langle \partial_{\bfnu}u^{+},u_{\mrm{s}}  \rangle_{\Gamma}\\[4pt]
 & = & \langle A\partial_{\bfnu}u^{-},u_{\mrm{i}}-u_{\mrm{s}}  \rangle_{\Gamma}-\langle \partial_{\bfnu}u^{-}_{\mrm{i}},u_{\mrm{i}}  \rangle_{\Gamma} + \langle \partial_{\bfnu}u^{+}_{\mrm{s}},u_{\mrm{s}}  \rangle_{\Gamma}.
\end{array}
\end{equation}
In the following arguments, we shall write $F|\nabla \varphi|^2$ instead of $F\nabla \varphi\cdot\nabla \varphi$ when $F$ is a matrix and $\varphi$ is a function. Integrating by parts, we find from (\ref{estimEnerTotal1}), (\ref{estimEnerTotal2}) 
\begin{equation}\label{estimEnerTotal3}
\begin{array}{ll}
  & c_d^{-1}\,\alpha^{\top}\mA(i\kappa)\,\alpha \\[10pt]
= & -\dsp \int_{\mathcal{D}} A \nabla u\cdot\nabla(u_{\mrm{i}}-u_{\mrm{s}}) +\kappa^2 \rho\,u\,(u_{\mrm{i}}-u_{\mrm{s}})\,d\bfx+\dsp \int_{\mathcal{D}}|\nabla u_{\mrm{i}}|^2+\kappa^2 u_{\mrm{i}}^2\,d\bfx+\dsp \int_{\Om}|\nabla u_{\mrm{s}}|^2+\kappa^2 u_{\mrm{s}}^2\,d\bfx\\[10pt]
= & \dsp\int_{\R^d} A|\nabla u_{\mrm{s}}|^2+\kappa^2 \rho\,u_{\mrm{s}}^2\,d\bfx+\int_{\mathcal{D}} (\mrm{Id}-A)|\nabla u_{\mrm{i}}|^2+\kappa^2(1-\rho) u_{\mrm{i}}^2\,d\bfx.
\end{array}
\end{equation}
The above identity will allow us to conclude when $\mrm{Id}-A$ is nonnegative definite and $1-\rho$ is nonnegative. To handle the case where $A-\mrm{Id}$ is nonnegative definite and $\rho-1$ is nonnegative, now we prove a second energy identity. Starting like in (\ref{estimEnerTotal2}), we write
\begin{equation}\label{estimEnerTotal4}
\begin{array}{lcl}
\langle \partial_{\bfnu}u^{+}_{\mrm{s}},u_{\mrm{i}} \rangle_{\Gamma}-\langle \partial_{\bfnu}u^{+}_{\mrm{i}},u_{\mrm{s}}  \rangle_{\Gamma} & = &  \langle \partial_{\bfnu}u^{+}_{\mrm{s}},u\rangle_{\Gamma}-\langle \partial_{\bfnu}u^{+}_{\mrm{s}},u_{\mrm{s}} \rangle_{\Gamma}-\langle \partial_{\bfnu}u^{+}_{\mrm{i}},u_{\mrm{s}}  \rangle_{\Gamma}\\[4pt]
 & = &  \langle A\partial_{\bfnu}u^{-},u\rangle_{\Gamma}-\langle \partial_{\bfnu}u^{-}_{\mrm{i}},u+u_{\mrm{s}}\rangle_{\Gamma}-\langle \partial_{\bfnu}u^{+}_{\mrm{s}},u_{\mrm{s}} \rangle_{\Gamma}.
\end{array}
\end{equation}
From this expression, using (\ref{estimEnerTotal1}), we deduce
\begin{equation}\label{estimEnerTotal5}
\begin{array}{ll}
  & c_d^{-1}\,\alpha^{\top}\mA(i\kappa)\,\alpha \\[10pt]
= & -\dsp \int_{\mathcal{D}}A|\nabla u|^2+\kappa^2\rho\,u^2\,d\bfx+\dsp \int_{\mathcal{D}} \nabla u_{\mrm{i}}\cdot\nabla(u+u_{\mrm{s}}) +\kappa^2 u_{\mrm{i}}\,(u+u_{\mrm{s}})\,d\bfx-\dsp \int_{\Om}|\nabla u_{\mrm{s}}|^2+\kappa^2 u_{\mrm{s}}^2\,d\bfx\\[10pt]
= & -\dsp\int_{\R^d} |\nabla u_{\mrm{s}}|^2+\kappa^2 u_{\mrm{s}}^2\,d\bfx - \int_{\mathcal{D}} (A-\mrm{Id})|\nabla u|^2+\kappa^2(\rho-1) u^2\,d\bfx.
\end{array}
\end{equation}
The previous analysis leads us to formulate the two assumptions:
\begin{Assumption}\label{Assumption1}
In $\mathcal{D}$, $\mrm{Id}-A$ is nonnegative definite and $1-\rho$ is nonnegative. Moreover, there exists a constant $C$ and a non empty open set $\varpi\subset\mathcal{D}$ on which there hold $0< C|\zeta|^2 \le (\mrm{Id}-A)\zeta\cdot\zeta$  for all $\zeta\in\R^d$ or $ 0< C \le 1-\rho$.
\end{Assumption}
\begin{Assumption}\label{Assumption2}
In $\mathcal{D}$, $A-\mrm{Id}$ is nonnegative definite and $\rho-1$ is nonnegative. Moreover, there exists a constant $C$ and a non empty open set $\varpi\subset\mathcal{D}$ on which there hold $0< C|\zeta|^2 \le (A-\mrm{Id})\zeta\cdot\zeta$  for all $\zeta\in\R^d$ or $ 0< C \le \rho-1$.
\end{Assumption}
\noindent From relations (\ref{estimEnerTotal3}) and (\ref{estimEnerTotal5}), we obtain the following result.
\begin{proposition}\label{PropoPenetrableWaveguide}\textsc{(Penetrable inclusion) -- Assume that $A$ and $\rho$ verify Assumption \ref{Assumption1} or Assumption \ref{Assumption2}. Then the set of non-scattering wavenumbers for Problem (\ref{PbChampTotalFreeSpaceIntrod}) is either empty or discrete.}
\end{proposition}
\noindent Using Equalities (\ref{estimEnerTotal3}) and (\ref{estimEnerTotal5}), we can provide other criteria, analogous to the ones derived in the study of the continuous interior transmission eigenvalue problem (see \cite{BoCH11,LaVa12,BoCh13}), which only involve the sign of $A-\mrm{Id}$ and $\rho-1$ in $\mathcal{D}\cap\mathcal{V}$ where $\mathcal{V}$ is a neighbourhood of $\Gamma=\partial\mathcal{D}$. To derive such criteria, we first prove a lemma which is very close to \cite[Proposition 2.1]{Sylv12}. For $\delta>0$, define $\mathcal{D}_{\delta}:=\{\bfx\in \mathcal{D}\,|\,\mrm{dist}(\bfx,\partial\mathcal{D})<\delta\}$. 
\begin{lemma}\label{lemmaSylv}
Let $\tilde{A}\in\mrm{L}^{\infty}(\mathcal{D},\R^{d\times d})$ and $\tilde{\rho}\in \mrm{L}^{\infty}(\mathcal{D},\R)$ be two parameters that satisfy  $\tilde{A}(\bfx)\zeta\cdot\zeta\ge C |\zeta|^2$ for all $\bfx\in\mathcal{D}$, $\zeta\in\R^d$ and $\tilde{\rho}(\bfx)\ge C$ for all $\bfx\in\mathcal{D}$, for some constant $C>0$. Assume that $\tilde{A}(\bfx)$ is symmetric for all $\bfx\in\mathcal{D}$. Consider some $\delta>0$. If $v\in\mH^1(\mathcal{D})$ verifies 
\begin{equation}\label{eqnInter}
-\div(\tilde{A}\nabla v) + \kappa^2\tilde{\rho}\,v=0,
\end{equation}
then there exists a constant $c>0$ independent of $v$, $\kappa$ such that, for $\kappa>0$ large enough,
\begin{equation}\label{estimSylv}
e^{c\kappa}\dsp\int_{\mathcal{D}\setminus\overline{\mathcal{D}_{\delta}}} |\nabla v|^2+\kappa^2 v^2\,d\bfx \le \dsp\int_{\mathcal{D}_{\delta}} v^2\,d\bfx.
\end{equation}
\end{lemma}
\begin{remark}
In particular, the result of this lemma ensures that for large values of $\kappa$, functions satisfying (\ref{eqnInter}) are more and more localized in a neighbourhood of $\Gamma$. 
\end{remark}
\begin{proof}
We adopt an approach which is used for example in \cite[\S3.2]{CaDN10}. For $\delta>0$, set $\Phi_{\delta}(\bfx)=\min(\mrm{dist}(\bfx,\partial\mathcal{D}),\delta)$. Introduce $c_1$, $c_2$ two positive constants such that $\tilde{A}(\bfx)\zeta\cdot\zeta\le c_1 |\zeta|^2$ for all $\bfx\in\mathcal{D}_{\delta}$, $\zeta\in\R^d$ and $\tilde{\rho}(\bfx)\ge c_2$ for all $\bfx\in\mathcal{D}_{\delta}$. Define the function $E$ such that $E(\bfx)=e^{\eps\kappa \Phi_{\delta}(\bfx)}-1$ with $\eps=\sqrt{c_2/c_1}$ (this \textit{ad hoc} value for $\eps$ will be needed in (\ref{relationInter2})). It is known that for a Lipschitz boundary $\partial \mathcal{D}$, $\Phi_{\delta}$ is an element of $\mrm{L}^{\infty}(\mathcal{D})$ such that $\nabla \Phi_{\delta}\in(\mrm{L}^{\infty}(\mathcal{D}))^2$. Moreover, there holds $|\nabla \Phi_{\delta}|\le 1$ on $\mathcal{D}$. Multiplying (\ref{eqnInter}) by $E^2 v$, integrating by parts and noticing that $E^2 v=0$ on $\partial\mathcal{D}$, we find
\begin{equation}\label{eqnAfterIpp}
\int_{\mathcal{D}} \tilde{A}\nabla v\cdot\nabla(E^2 v)+\kappa^2\tilde{\rho}\,(E v)^2\,d\bfx=0.
\end{equation}
Writing 
\[
\tilde{A}\nabla v\cdot\nabla(E^2 v) = E^2 \tilde{A}\nabla v\cdot\nabla v + 2 E v\tilde{A}\nabla v\cdot\nabla E
\]
and 
\[
\tilde{A}\nabla(E v)\cdot\nabla(E v) = E^2 \tilde{A}\nabla v\cdot\nabla v + v^2\tilde{A}\nabla E\cdot\nabla E+2 E v\tilde{A}\nabla v\cdot\nabla E,
\]
we deduce 
\[
\tilde{A}\nabla v\cdot\nabla(E^2 v) = \tilde{A}\nabla(E v)\cdot\nabla(E v) - v^2\tilde{A}\nabla E\cdot\nabla E.
\]
Plugging this equality in (\ref{eqnAfterIpp}), we get the identity
\begin{equation}\label{identity}
\dsp\int_{\mathcal{D}} \tilde{A}\nabla(E v)\cdot\nabla(E v)+\kappa^2\tilde{\rho} (E v)^2\,d\bfx =
\dsp\int_{\mathcal{D}}v^2\tilde{A}\nabla E\cdot\nabla E\,d\bfx.
\end{equation}
Since $E$ is constant on $\mathcal{D}\setminus\overline{\mathcal{D}_{\delta}}$, from (\ref{identity}) we obtain
\begin{equation}\label{relationInter1}
\dsp\int_{\mathcal{D}\setminus\overline{\mathcal{D}_{\delta}}} \tilde{A}\nabla(E v)\cdot\nabla(E v)+\kappa^2\tilde{\rho} (E v)^2\,d\bfx \le
\dsp\int_{\mathcal{D}_{\delta}}v^2(\tilde{A}\nabla E\cdot\nabla E-\kappa^2\tilde{\rho} E^2)\,d\bfx.
\end{equation}
With our particular choice for $\eps$, on $\mathcal{D}_{\delta}$ we can write
\begin{equation}\label{relationInter2}
(\tilde{A}\nabla E\cdot\nabla E-\kappa^2\tilde{\rho} E^2)\le \kappa^2(\eps^2 c_1e^{2\eps\kappa \Phi_{\delta}(\bfx)}-c_2\,(e^{\eps\kappa \Phi_{\delta}(\bfx)}-1)^2) \le 2\,c_2\,\kappa^2 e^{\eps\kappa \delta}.
\end{equation}
Using (\ref{relationInter2}) in (\ref{relationInter1}) yields
\[
\dsp\frac{(e^{\eps\kappa \delta}-1)^2}{2\,c_2\,\kappa^2 e^{\eps\kappa \delta}}\int_{\mathcal{D}\setminus\overline{\mathcal{D}_{\delta}}} \tilde{A}\nabla v\cdot\nabla v+\kappa^2\tilde{\rho}  v^2\,d\bfx \le
\dsp\int_{\mathcal{D}_{\delta}}v^2\,d\bfx,
\]
which leads to (\ref{estimSylv}) for $\kappa$ large enough. 
\end{proof}
\noindent Now, we can use Lemma \ref{lemmaSylv} to localize to $\Gamma$ the assumptions made on $A$ and $\rho$.
\begin{Assumption}\label{Assumption3}
There is a neighbourhood of $\Gamma$, denoted $\mathcal{V}$, such that $\mrm{Id}-A$ is nonnegative definite on $\mathcal{D}\cap\mathcal{V}$. Moreover, there exists a constant $C$ such that there holds $0<C \le 1-\rho$ on $\mathcal{D}\cap\mathcal{V}$.
\end{Assumption}
\begin{Assumption}\label{Assumption4}
There is a neighbourhood of $\Gamma$, denoted $\mathcal{V}$, such that $A-\mrm{Id}$ is nonnegative definite on $\mathcal{D}\cap\mathcal{V}$. Moreover, there exists a constant $C$ such that there holds $0<C \le \rho-1$ on $\mathcal{D}\cap\mathcal{V}$.
\end{Assumption}
\begin{proposition}\label{PropoPenetrableWaveguideBis}\textsc{(Penetrable inclusion) -- Assume that $A$ and $\rho$ verify Assumption \ref{Assumption3} or Assumption \ref{Assumption4}. Then the set of non-scattering wavenumbers for Problem (\ref{PbChampTotalFreeSpaceIntrod}) is either empty or discrete.}
\end{proposition}
\begin{proof}
Assume that $A$ and $\rho$ verify Assumption \ref{Assumption3}. Introduce $\delta>0$ small enough so that the set $\mathcal{D}_{\delta}=\{\bfx\in \mathcal{D}\,|\,\mrm{dist}(\bfx,\partial\mathcal{D})<\delta\}$ verifies $\mathcal{D}_{\delta}\subset \mathcal{V}$. If $\alpha=(\alpha_1,\dots,\alpha_N)^{\top}$ is an element of $\ker\,\mA(i\kappa)$ then, according to (\ref{estimEnerTotal3}), the field $u_{\mrm{i}}=\sum_{n=1}^N\alpha_n e^{-\kappa\boldsymbol{\theta}_n\cdot\boldsymbol{x}}$ and the associated scattered field $u_{\mrm{s}}$ must verify
\begin{equation}\label{eqnInterTer}
0=\dsp\int_{\R^d} A|\nabla u_{\mrm{s}}|^2+\kappa^2 \rho\,u_{\mrm{s}}^2\,d\bfx+\int_{\mathcal{D}} (\mrm{Id}-A)|\nabla u_{\mrm{i}}|^2+\kappa^2(1-\rho) u_{\mrm{i}}^2\,d\bfx.
\end{equation}
The boundedness of $A$, $\rho$ and Lemma \ref{lemmaSylv} applied to $u_{\mrm{i}}$ (with $\tilde{A}=\mrm{Id}$ and $\tilde{\rho}=1$) allow to write \begin{equation}\label{eqnInterBis}
\begin{array}{ll}
 & \dsp\int_{\mathcal{D}} (\mrm{Id}-A)|\nabla u_{\mrm{i}}|^2+\kappa^2(1-\rho) u_{\mrm{i}}^2\,d\bfx \\[10pt]
\ge & \dsp\int_{\mathcal{D}_{\delta}} (\mrm{Id}-A)|\nabla u_{\mrm{i}}|^2+\kappa^2(1-\rho) u_{\mrm{i}}^2\,d\bfx - c\,\dsp\int_{\mathcal{D}\setminus\overline{\mathcal{D}_{\delta}}} |\nabla u_{\mrm{i}}|^2+\kappa^2 u_{\mrm{i}}^2\,d\bfx \\[10pt]
\ge & \dsp\int_{\mathcal{D}_{\delta}} (\mrm{Id}-A)|\nabla u_{\mrm{i}}|^2+\kappa^2(1-\rho) u_{\mrm{i}}^2\,d\bfx - c\,\dsp\int_{\mathcal{D}_{\delta}} u_{\mrm{i}}^2\,d\bfx\ , \\[10pt]
\end{array}
\end{equation}
where $c>0$ is a constant independent of $\kappa$ and $\alpha$ which may change from one line to another. Plugging (\ref{eqnInterBis}) in (\ref{eqnInterTer}) and using that $0<C\le 1-\rho$ on $\mathcal{D}_{\delta}$, we deduce 
\[
0 \ge c\,\kappa^2\dsp\int_{\mathcal{D}_{\delta}}u_{\mrm{i}}^2\,d\bfx
\]
for $\kappa$ large enough. This implies $u_{\mrm{i}}=0$ on $\mathcal{D}_{\delta}$, which is possible if and only if $\alpha$ is the null vector. Therefore, $\mA(i\kappa)$ is injective for $\kappa$ large enough and the analyticity of the map $k\mapsto \det\mA(k)$ leads to the conclusion of the proposition. The case where $A$ and $\rho$ verify Assumption \ref{Assumption4} can be treated similarly using Lemma \ref{lemmaSylv} applied to $u$ with $\tilde{A}=A$ and $\tilde{\rho}=\rho$.
\end{proof}
\begin{remark}
As mentioned previously, the technique to prove Proposition \ref{PropoPenetrableWaveguideBis} is directly inspired by what has been done to consider the continuous interior transmission eigenvalue problem. However, in the discrete setting, Assumption \ref{Assumption3} and Assumption \ref{Assumption4} may probably be relaxed. Thus, one could imagine that imposing conditions on the physical parameters $A$ and $\rho$ only in some particular regions of $\Gamma$, associated with the directions of the incident waves, is enough to obtain the result of Proposition \ref{PropoPenetrableWaveguideBis}. This is obvious when there is only one incident direction. And when there are more and more incident directions, one would retrieve the conditions of Assumption \ref{Assumption3} and Assumption \ref{Assumption4}. But obtaining such a criterion for a given finite number of incident directions is still an open problem.
\end{remark}

\subsection{Open questions}
Around these questions, we can formulate a series of problems we do not know how to solve. Can we prove that the set of non-scattering wavenumbers is discrete or empty when the scattering directions are not equal to the opposite of the incident directions? Can we show that the non-scattering wavenumbers for Problem (\ref{PbChampTotalFreeSpace}) (resp. (\ref{PbChampTotalFreeSpaceIntrod})) do not accumulate at $0$? Can we relax Assumptions \ref{Assumption1}, \ref{Assumption2}, \ref{Assumption3} and \ref{Assumption4}? Do the non-scattering wavenumbers in this discrete setting (if they exist) converge to the transmission eigenvalues of the continuous framework when the number of incident and scattering directions tend to $+\infty$?

\section{Far field invisibility}\label{FFInvisibility}
In this section, we change the point of view. Let us consider $k>0$ a given wavenumber and $\mathcal{D}$ a given domain. We want to construct an inclusion supported in $\overline{\mathcal{D}}$ for which the relative scattering matrix $\mA(k)$ defined in (\ref{defRelScaMat}) is null. We assume that the physical coefficient $A$ verifies $A=\mrm{Id}$ in $\R^d$ so that the material of the inclusion is characterized only by the real valued parameter $\rho$. Then, the scattering of the incident field $u_{\mrm{i}}=\sum_{n=1}^{N}\alpha_n e^{i k \boldsymbol{\theta}_n\cdot\boldsymbol{x}}$ is described by the problem
\begin{equation}\label{PbChampTotalFreeSpaceBis}
\begin{array}{|rcll}
\multicolumn{4}{|l}{\mbox{Find }u\in\mH^1_{\loc}(\R^d)\mbox{ such that}}\\[4pt]
-\Delta u & = & k^2 \rho\,u & \mbox{ in }\R^d,\\
u & = & u_{\mrm{i}}+u_{\mrm{s}} & \mbox{ in }\R^d,\\[4pt]
\multicolumn{4}{|c}{\dsp\lim_{r\to +\infty} r^{\frac{d-1}{2}}\left( \frac{\partial u_{\mrm{s}}}{\partial r}-ik u_{\mrm{s}} \right)  = 0.}
\end{array}
\end{equation}
We search for $\rho$ under the form  $\rho=1+\eps\mu$ where $\eps>0$ is a parameter small with respect to $1$ and where $\mu\in \mrm{L}^{\infty}(\R^d,\R)$ is a function supported in $\overline{\mathcal{D}}$. We emphasize that the inclusion we wish to create is a small perturbation of the reference material. This is a key element of the approach we will follow. The technique we will use has been introduced in \cite{Naza11,Naza11c,Naza12,Naza13,CaNR12,Naza11b} with the concept of ``enforced stability for embedded eigenvalues''. In these works, the authors develop a method for constructing small regular and singular perturbations of a waveguide that preserve the multiplicity of the point spectrum on a given interval of the continuous spectrum. The approach has been adapted in \cite{BoNa13} (see also \cite{BoNTSu} for an application to a water wave problem) to prove the existence of regular perturbations of a waveguide, for which several waves at given frequencies pass through without any distortion or with only a phase shift.
\subsection{One incident direction and $N$ scattering directions}\label{paragraphInvisi}
To simplify the presentation of the method, we first assume that there is only one incident direction $\boldsymbol{\theta}_{\mrm{i}}$ (\textit{i.e.} in (\ref{PbChampTotalFreeSpaceBis}), we take  $u_{\mrm{i}}=e^{i k \boldsymbol{\theta}_{\mrm{i}}\cdot\bfx}$) and $N$ scattering directions $\boldsymbol{\theta}_1,\dots,\boldsymbol{\theta}_N$. For a given $\eps>0$, we denote $u^{\eps}$ the solution of Problem (\ref{PbChampTotalFreeSpaceBis}) with $\rho=1+\eps \mu$. We proceed to a formal asymptotic expansion of $u^{\eps}$. This function admits the decomposition $u^{\eps}=u_{\mrm{i}}+u_{\mrm{s}}^{\eps}$ where $u_{\mrm{s}}^{\eps}$ corresponds to the field scattered by the inclusion. We take the simplest ansatz for a regular perturbation of a partial differential equation \cite{Kato66,MaNP00} 
\[
u_{\mrm{s}}^{\eps}=0+\eps \hat{u} +\eps^2 \tilde{u}+\dots,
\]
where the correction terms $\hat{u}$, $\tilde{u}$ have to be determined and where the dots stand for higher order terms insignificant in our asymptotic analysis. Let us point out that we choose an ansatz for the scattered field which vanishes at the zero order because $\rho=1+\eps\mu$ is a perturbation of order $\eps$ of the parameter of the reference material (which does not produce any scattered field). Plugging the expression of $u^{\eps}$ and the expansion of $u_{\mrm{s}}^{\eps}$ in (\ref{PbChampTotalFreeSpaceBis}), we find that $\hat{u}$ and $\tilde{u}$ are uniquely determined as the solutions of the problems 
\begin{equation}\label{ChampScaOrdres}
\begin{array}{|rcll}
\multicolumn{4}{|l}{\mbox{Find }\hat{u}\in\mH^1_{\loc}(\R^d)\mbox{ such that}}\\[4pt]
-\Delta \hat{u}-k^2\hat{u} & = & k^2 \mu u_{\mrm{i}} & \mbox{ in }\R^d,\\
\multicolumn{4}{|c}{\dsp\lim_{r\to +\infty} r^{\frac{d-1}{2}}\left( \frac{\partial \hat{u}}{\partial r}-ik \hat{u} \right)  = 0}
\end{array}\qquad\qquad\begin{array}{|rcll}
\multicolumn{4}{|l}{\mbox{Find }\tilde{u}\in\mH^1_{\loc}(\R^d)\mbox{ such that}}\\[4pt]
-\Delta \tilde{u}-k^2\,\tilde{u} & = & k^2 \mu \hat{u} & \mbox{ in }\R^d,\\
\multicolumn{4}{|c}{\dsp\lim_{r\to +\infty} r^{\frac{d-1}{2}}\left( \frac{\partial \tilde{u}}{\partial r}-ik \tilde{u} \right)  = 0.}
\end{array}
\end{equation}
From (\ref{FFPtermeSource}), we deduce that the far field pattern of $u_{\mrm{s}}^{\eps}$ in the direction of observation $\boldsymbol{\theta}_{\mrm{s}}$, denoted $u_{\mrm{s}}^{\eps\,\infty}(\boldsymbol{\theta}_{\mrm{s}},\boldsymbol{\theta}_{\mrm{i}})$, admits the asymptotic expansion
\begin{equation}\label{FFPOrdre2}
u_{\mrm{s}}^{\eps\,\infty}(\boldsymbol{\theta}_{\mrm{s}},\boldsymbol{\theta}_{\mrm{i}}) =0+ \eps\,c_d\,k^2\dsp \int_{\mathcal{D}} \mu\,e^{ik(\boldsymbol{\theta}_{\mrm{i}}-\boldsymbol{\theta}_{\mrm{s}})\cdot\bfx}\,d\bfx +  \eps^2\,c_d\,k^2\dsp \int_{\mathcal{D}} \mu\,\hat{u}\,e^{-ik \boldsymbol{\theta}_{\mrm{s}}\cdot\bfx}\,d\bfx+\dots\ .
\end{equation}
Observing (\ref{FFPOrdre2}), we see it is easy to find functions $\mu$ such that $u_{\mrm{s}}^{\eps\,\infty}(\boldsymbol{\theta}_n,\boldsymbol{\theta}_{\mrm{i}})$ vanishes at order $\eps$ for $n=1,\dots,N$. However, this is not sufficient since we want to impose $u_{\mrm{s}}^{\eps\,\infty}(\boldsymbol{\theta}_n,\boldsymbol{\theta}_{\mrm{i}})=0$ (at any order in $\eps$). To control the higher order terms in $\eps^2$, $\eps^3$,$\dots$ whose dependence with respect to $\mu$ is less simple than for the first term of the asymptotics, we will use the fixed point theorem. To obtain a fixed point formulation, we look for $\mu$ under the form 
\begin{equation}\label{ExpressMu}
\mu=\mu_0+\dsp\sum_{n=1}^N \tau_{1,n}\,\mu_{1,n} + \dsp\sum_{n=1}^N \tau_{2,n}\, \mu_{2,n}.
\end{equation}
In this expression, $\tau_{1,n}$, $\tau_{2,n}$ are real parameters that we will tune to achieve invisibility in the directions $\boldsymbol{\theta}_1,\dots,\boldsymbol{\theta}_N$. We need $2N$ real parameters because we want to cancel $N$ complex coefficients. Moreover, in (\ref{ExpressMu}) $\mu_{0}$, $\mu_{1,n}$, $\mu_{2,n}$ are given real valued functions, supported in  $\overline{\mathcal{D}}$, verifying
\begin{equation}\label{RelOrtho}
\begin{array}{lcllcl}
\dsp\int_{\mathcal{D}} \mu_{0}\cos(k(\boldsymbol{\theta}_{\mrm{i}}-\boldsymbol{\theta}_{n'})\cdot\boldsymbol{x}))\,d\boldsymbol{x} &= & 0,\qquad & \qquad\dsp\int_{\mathcal{D}} \mu_{0}\sin(k(\boldsymbol{\theta}_{\mrm{i}}-\boldsymbol{\theta}_{n'})\cdot\boldsymbol{x}))\,d\boldsymbol{x} &= & 0\\[12pt]
\dsp\int_{\mathcal{D}} \mu_{1,n}\cos(k(\boldsymbol{\theta}_{\mrm{i}}-\boldsymbol{\theta}_{n'})\cdot\boldsymbol{x}))\,d\boldsymbol{x} &= & \delta^{nn'},\qquad & \qquad\dsp\int_{\mathcal{D}} \mu_{1,n}\sin(k(\boldsymbol{\theta}_{\mrm{i}}-\boldsymbol{\theta}_{n'})\cdot\boldsymbol{x}))\,d\boldsymbol{x} &= & 0\\[12pt]
\dsp\int_{\mathcal{D}} \mu_{2,n}\cos(k(\boldsymbol{\theta}_{\mrm{i}}-\boldsymbol{\theta}_{n'})\cdot\boldsymbol{x}))\,d\boldsymbol{x} &= & 0,\qquad & \qquad \dsp\int_{\mathcal{D}} \mu_{2,n}\sin(k(\boldsymbol{\theta}_{\mrm{i}}-\boldsymbol{\theta}_{n'})\cdot\boldsymbol{x}))\,d\boldsymbol{x} &= & \delta^{nn'}
\end{array}
\end{equation}
for all $n,n'=1,\dots,N$. In (\ref{RelOrtho}), $\delta^{nn'}$ denotes the Kronecker delta such that $\delta^{nn'}=1$ if $n=n'$ and $\delta^{nn'}=0$ else. At this stage, it is important to assume that there holds $\boldsymbol{\theta}_{\mrm{i}}-\boldsymbol{\theta}_{n'}\ne0$ for $n'=1,\dots, N$ so that we can indeed find functions $\mu_{2,n}$ which satisfy the last equation of (\ref{RelOrtho}). In \S\ref{SubSectionNumExpe}, dedicated to numerical experiments, we will explain how to build explicit functions $\mu_{1,n}$, $\mu_{2,n}$ verifying (\ref{RelOrtho}). With this particular choice for $\mu$, plugging (\ref{ExpressMu}) in (\ref{FFPOrdre2}) and using (\ref{RelOrtho}), we obtain, for $n=1,\dots,N$,  the expansion
\begin{equation}\label{FFPTotal}
u_{\mrm{s}}^{\eps\,\infty}(\boldsymbol{\theta}_n,\boldsymbol{\theta}_{\mrm{i}}) = \eps\,c_d\,k^2\,(\tau_{1,n}+i\tau_{2,n})+\eps^2\,c_d\,k^2\,(F^{\eps}_{1,n}(\tau)+iF^{\eps}_{2,n}(\tau)),
\end{equation}
where $F^{\eps}_{1,n}$, $F^{\eps}_{2,n}$ denote some functions (whose precise definition can be found in (\ref{FFexpression1})) of $\eps$ and $\tau:= (\tau_{1,1},\dots,\tau_{1,N},\tau_{2,1},\dots,\tau_{2,N})^{\top}$. Now, to impose $u_{\mrm{s}}^{\eps\,\infty}(\boldsymbol{\theta}_n,\boldsymbol{\theta}_{\mrm{i}})=0$ for $n=1,\dots,N$, we see from (\ref{FFPTotal}) that we have to solve the problem \\
\begin{equation}\label{PbFixedPoint}
\begin{array}{|l}
\mbox{Find }\tau\in\R^{2N}\mbox{ such that }\tau = F^{\eps}(\tau),
\end{array}~\\[5pt]
\end{equation}
with 
\begin{equation}\label{DefMapContra}
F^{\eps}(\tau):= -\eps\,(F^{\eps}_{1,1}(\tau),\dots,F^{\eps}_{1,N}(\tau),F^{\eps}_{2,1}(\tau),\dots,F^{\eps}_{2,N}(\tau))^{\top}. 
\end{equation}
Lemma \ref{lemmaTech} hereafter ensures that for any given parameter $\gamma>0$, the map $F^{\eps}$ is a contraction of $\mathbb{B}_\gamma:=\{\tau\in\R^{2N}\,\big|\,|\tau| \le \gamma\}$ for $\eps$ small enough. Therefore, the Banach fixed-point theorem guarantees the existence of some $\eps_0>0$ such that for all $\eps\in(0;\eps_0]$, Problem (\ref{PbFixedPoint}) has a unique solution in $\mathbb{B}_\gamma$.\\
\newline
It is important to observe that because of the orthogonality conditions (\ref{RelOrtho}), we are sure that the constructed function $\mu$ defined by (\ref{ExpressMu}) verifies $\mu\not\equiv0$ when $\mu_0\not\equiv0$. As a consequence, we indeed defined a non trivial perturbation of the reference medium, supported in $\overline{\mathcal{D}}$, which is such that the far field pattern of the scattered field associated with the incident plane wave $u_{\mrm{i}}=e^{i k \boldsymbol{\theta}_{\mrm{i}}\cdot\boldsymbol{x}}$, vanishes in the directions $\boldsymbol{\theta}_1,\dots,\boldsymbol{\theta}_N$. \\
\newline
Let us summarize this result in the following statement.
\begin{proposition}
Let $k>0$ be a given wavenumber and $\mathcal{D}$ refer to a given Lipschitz domain. Consider $\boldsymbol{\theta}_{\mrm{i}},\boldsymbol{\theta}_1,\dots,\boldsymbol{\theta}_N$ $N+1$ elements of $\mathbb{S}^{d-1}$ such that $\boldsymbol{\theta}_n\ne \boldsymbol{\theta}_{\mrm{i}}$ for $n=1,\dots,N$. Define the incident plane wave $u_{\mrm{i}}:=e^{i k \boldsymbol{\theta}_{\mrm{i}}\cdot\boldsymbol{x}}$. Then, there exists a non trivial parameter $\rho$, with $\rho-1$ supported in $\overline{\mathcal{D}}$, such that the far field pattern of the scattered field defined by (\ref{PbChampTotalFreeSpaceBis}) vanishes in the directions $\boldsymbol{\theta}_1,\dots,\boldsymbol{\theta}_N$.
\end{proposition}

\noindent In the following lemma, we show some properties of the operator $F^{\eps}$ that we used in the previous analysis.
\begin{lemma}\label{lemmaTech}
Let $\gamma>0$ be a given parameter. Then, there exists $\eps_0>0$ such that for all $\eps\in(0;\eps_0]$, the map $F^{\eps}$ is a contraction of $\mathbb{B}_\gamma=\{\tau\in\R^{2N}\,\big|\,|\tau| \le \gamma\}$.
\end{lemma}
\begin{proof}
We assume that $\gamma>0$ is given. Decomposing $u_{\mrm{s}}^{\eps}$ as $u_{\mrm{s}}^{\eps}=0+\eps \hat{u} +\eps^2 \check{u}^{\eps}$, where $\hat{u}$ is defined in (\ref{ChampScaOrdres}), we obtain that $\check{u}^{\eps}$ must be the solution of the problem 
\begin{equation}\label{defTermRaccour}
\begin{array}{|rcll}
\multicolumn{4}{|l}{\mbox{Find }\check{u}^{\eps}\in\mH^1_{\loc}(\R^d)\mbox{ such that}}\\[4pt]
-\Delta \check{u}^{\eps}-k^2\,\rho\,\check{u}^{\eps} & = & k^2 \mu \hat{u} & \mbox{ in }\R^d,\\
\multicolumn{4}{|c}{\dsp\lim_{r\to +\infty} r^{\frac{d-1}{2}}\left( \frac{\partial \check{u}^{\eps}}{\partial r}-ik \check{u}^{\eps} \right)  = 0.}
\end{array}
\end{equation}
Therefore, using (\ref{FFPtermeSource}), we find that the functions $F^{\eps}_{1,n}$,  $F^{\eps}_{2,n}$ introduced in (\ref{FFPTotal}) verify
\begin{equation}\label{FFexpression1}
 F^{\eps}_{1,n}(\tau) = \Re e\left(\dsp \int_{\mathcal{D}}
\mu( \hat{u}+\eps \check{u}^{\eps})\,e^{-ik \boldsymbol{\theta}_n\cdot\boldsymbol{x}}\,d\bfx\right) \mbox{ and }\ F^{\eps}_{2,n}(\tau) = \Im m\left(\dsp \int_{\mathcal{D}} \mu\,( \hat{u}+\eps \check{u}^{\eps})\,e^{-ik \boldsymbol{\theta}_n\cdot\boldsymbol{x}}\,d\bfx\right).
\end{equation}
General results of perturbations theory for linear operators (see \cite[Chap. 7]{Kato66}, \cite[Chap. 4]{HiPh57}) yield, for $\eps$ small enough, the continuity estimate 
\begin{equation}\label{ContinuityParam}
\|\hat{u}-\hat{u}'\|_{\mathcal{D}}+\|\check{u}^{\eps}-\check{u}^{\eps}{}'\|_{\mathcal{D}} \le C\,|\tau-\tau'|,\qquad \forall\tau,\tau'\in\mathbb{B}_\gamma.
\end{equation}
Here, $C>0$ is a constant (which can change from one line to another) independent of $\eps$ while $\hat{u}'$, $\check{u}^{\eps}{}'$ denote respectively the solutions of (\ref{ChampScaOrdres}), (\ref{defTermRaccour}) with $\tau$ replaced by $\tau'$ in the definition of $\mu$ (\ref{ExpressMu}). Using (\ref{ContinuityParam}) in (\ref{FFexpression1}), we deduce $|F^{\eps}_{1,n}(\tau)-F^{\eps}_{1,n}(\tau')|+|F^{\eps}_{2,n}(\tau)-F^{\eps}_{2,n}(\tau')| \le C\,|\tau-\tau'|$. This result and definition (\ref{DefMapContra}) then imply
\begin{equation}\label{EstimContra}
|F^{\eps}(\tau)-F^{\eps}(\tau')|\le C\,\eps\,|\tau-\tau'|,\qquad \forall\tau,\tau'\in\mathbb{B}_\gamma.
\end{equation}
Taking $\tau'=0$ in (\ref{EstimContra}) and remarking that $|F^{\eps}(0)| \le C\,\eps$ (use (\ref{FFexpression1}) and the conditions (\ref{RelOrtho}) imposed on $\mu_0$ to show the latter inequality), we find $|F^{\eps}(\tau)| \le C\,\eps$ for all $\tau\in\mathbb{B}_\gamma$. With (\ref{EstimContra}), this finally allows to conclude that the map $F^{\eps}$ is a contraction of $\mathbb{B}_\gamma$ for $\eps$ small enough.
\end{proof}
\noindent Let us denote $\tau\,{}^{\mrm{sol}}\in\mathbb{B}_\gamma$ the unique solution of Problem (\ref{PbFixedPoint}). The end of the previous proof ensures that there exists a constant $c_0>0$ independent of $\eps$ such that 
\begin{equation}\label{RelAPos}
|\tau\,{}^{\mrm{sol}}| = |F^{\eps}(\tau\,{}^{\mrm{sol}})| \le c_0\,\eps,\qquad \forall\eps\in(0;\eps_0].
\end{equation}
Estimate (\ref{RelAPos}) tells us that the function $\mu$ is equal to $\mu_0$ at first order.

\subsection{Backscattering measurements}

We come back to the study of the relative scattering matrix $\mA(k)\in\Cplx^{N\times N}$ associated with (\ref{PbChampTotalFreeSpaceBis}). We remind the reader that in the definition of $\mA(k)$, we assumed that the scattering directions are $-\boldsymbol{\theta}_1,\dots,-\boldsymbol{\theta}_N$ where  $\boldsymbol{\theta}_1,\dots,\boldsymbol{\theta}_N$ denote the incident directions. We wish to construct an inclusion of material for which $\mA(k)=0_{N\times N}$. It is well-known (see \cite[Theorem 8.8]{CoKr13}) that the far field pattern introduced in (\ref{scatteredFieldFreeSpaceIntro}) satisfies the reciprocity relation  
\[
u_{\mrm{s}}^{\infty}(\boldsymbol{\theta}_{\mrm{s}},\boldsymbol{\theta}_{\mrm{i}}) = u_{\mrm{s}}^{\infty}(-\boldsymbol{\theta}_{\mrm{i}},-\boldsymbol{\theta}_{\mrm{s}}),\qquad\forall \boldsymbol{\theta}_{\mrm{s}},\boldsymbol{\theta}_{\mrm{i}}\in\mathbb{S}^{d-1}.
\]
From the definition of $\mA(k)$ (see (\ref{defRelScaMat})) and this property, we deduce that 
\[
\mA_{mn}(k)=u_{\mrm{s}}^{\infty}(-\boldsymbol{\theta}_m,\boldsymbol{\theta}_n) = u_{\mrm{s}}^{\infty}(-\boldsymbol{\theta}_n,\boldsymbol{\theta}_m)=\mA_{nm}(k).
\]
Therefore, the matrix $\mA(k)$ is symmetric and we need to cancel only $N(N+1)/2$ complex terms. Following the approach of the previous section, we search for $\rho$ under the form  $\rho=1+\eps\mu$ with 
\begin{equation}\label{ExpressMuBis}
\mu=\mu_0+\dsp\sum_{n=1}^N\sum_{m=1}^n \tau_{1,m,n}\,\mu_{1,m,n} + \dsp\sum_{n=1}^N\sum_{m=1}^n \tau_{2,m,n}\, \mu_{2,m,n}.
\end{equation}
In this expression, $\tau_{1,m,n}$, $\tau_{1,m,n}$ denote some real parameters to fix while $\mu_{0}$, $\mu_{1,m,n}$, $\mu_{2,m,n}$ refer to given real valued functions, supported in $\overline{\mathcal{D}}$, verifying
\begin{equation}\label{RelOrthoBis}
\begin{array}{lcl}
\dsp\int_{\mathcal{D}} \mu_{0}\cos(k(\boldsymbol{\theta}_{n'}+\boldsymbol{\theta}_{m'})\cdot\boldsymbol{x}))\,d\boldsymbol{x} &= & 0,\\[12pt]

\dsp\int_{\mathcal{D}} \mu_{0}\sin(k(\boldsymbol{\theta}_{n'}+\boldsymbol{\theta}_{m'})\cdot\boldsymbol{x}))\,d\boldsymbol{x} &= & 0,\\[12pt]

\dsp\int_{\mathcal{D}} \mu_{1,m,n}\cos(k(\boldsymbol{\theta}_{n'}+\boldsymbol{\theta}_{m'})\cdot\boldsymbol{x}))\,d\boldsymbol{x} &= & \begin{array}{|l}1 \mbox{ if }(m,n)=(m',n')\mbox{ or }(m,n)=(n',m')\\ 0 \mbox{ else,} \end{array}\\[12pt]

\dsp\int_{\mathcal{D}} \mu_{1,m,n}\sin(k(\boldsymbol{\theta}_{n'}+\boldsymbol{\theta}_{m'})\cdot\boldsymbol{x}))\,d\boldsymbol{x} &= & 0,\\[12pt]

\dsp\int_{\mathcal{D}} \mu_{2,m,n}\cos(k(\boldsymbol{\theta}_{n'}+\boldsymbol{\theta}_{m'})\cdot\boldsymbol{x}))\,d\boldsymbol{x} &= & 0,\\[12pt]

\dsp\int_{\mathcal{D}} \mu_{2,m,n}\sin(k(\boldsymbol{\theta}_{n'}+\boldsymbol{\theta}_{m'})\cdot\boldsymbol{x}))\,d\boldsymbol{x} &= & \begin{array}{|l}1 \mbox{ if }(m,n)=(m',n')\mbox{ or }(m,n)=(n',m')\\ 0 \mbox{ else.} \end{array}
\end{array}
\end{equation}
The existence of such functions $\mu_{0}$, $\mu_{1,m,n}$, $\mu_{2,m,n}$ can be shown working as in \S\ref{SubSectionNumExpe} provided that the $N(N+1)/2$ elements of the family $\{\boldsymbol{\theta}_{n'}+\boldsymbol{\theta}_{m'}\}_{1\le m'\le n'\le N}$ are all non null and all different. With such a $\mu$, using (\ref{FFPtermeSource}) and working as in (\ref{FFPTotal}), we obtain, for $m,n=1,\dots,N$,  the expansion
\begin{equation}\label{FFPTotalBis}
u_{\mrm{s}}^{\eps\,\infty}(-\boldsymbol{\theta}_m,\boldsymbol{\theta}_n) = \eps\,c_d\,k^2\,(\tau_{1,m,n}+i\tau_{2,m,n})+\eps^2\,c_d\,k^2\,(G^{\eps}_{1,m,n}(\tau)+iG^{\eps}_{2,m,n}(\tau)).
\end{equation}
Here, $G^{\eps}_{1,m,n}$, $G^{\eps}_{2,m,n}$ denote some functions, defined as in (\ref{FFexpression1}), of $\eps$ and $\tau:=(\tau_1,\tau_2)$, where $\tau_1=(\tau_{1,m,n})_{N\times N}$, $\tau_2=(\tau_{2,m,n})_{N\times N}$. According to (\ref{FFPTotalBis}), to impose $u_{\mrm{s}}^{\eps\,\infty}(-\boldsymbol{\theta}_m,\boldsymbol{\theta}_n)=0$ for $m,n=1,\dots,N$, it just remains to solve the problem \\
\begin{equation}\label{PbFixedPointBis}
\begin{array}{|l}
\mbox{Find }\tau\in S^{N}\times S^{N}\mbox{ such that }\tau = G^{\eps}(\tau),
\end{array}~\\[5pt]
\end{equation}
where $S^N$ denotes the space of symmetric matrices. In this expression, the map $\tau\mapsto G^{\eps}$ is defined via 
\begin{equation}\label{DefMapContraBis}
G^{\eps}(\tau):= -\eps\,(G^{\eps}_{1}(\tau),G^{\eps}_{2}(\tau)),
\end{equation}
where $G^{\eps}_{1}(\tau)$, $G^{\eps}_{2}(\tau)$ stand for the $N\times N$ symmetric matrices made of the terms $G^{\eps}_{1,m,n}(\tau)$, $G^{\eps}_{1,m,n}(\tau)$ respectively. A simple adaptation of the proof of Lemma \ref{lemmaTech} allows to demonstrate that for any $\gamma>0$, there exists $\eps_0>0$ such that for all $\eps\in(0;\eps_0]$, the map $G^{\eps}$ is a contraction of $\{\tau\in S^{N}\times S^{N}\,\big|\,|\tau| \le \gamma\}$. This final result leads to the following statement.
\begin{proposition}
Let $k>0$ be a given wavenumber and $\mathcal{D}$ refer to a given Lipschitz domain. Consider $\boldsymbol{\theta}_1,\dots,\boldsymbol{\theta}_N$ $N$ vectors of $\mathbb{S}^{d-1}$ such that the $N(N+1)/2$ elements of the family $\{\boldsymbol{\theta}_{n'}+\boldsymbol{\theta}_{m'}\}_{1\le m'\le n'\le N}$ are all non null and all different. Then, there exists a non trivial parameter $\rho$, with $\rho-1$ supported in $\overline{\mathcal{D}}$, such that the relative scattering matrix $\mA(k)$ associated with (\ref{PbChampTotalFreeSpaceBis}) verifies $\mA(k)=0_{N\times N}$. As a consequence, for such inclusions, for any incident field combination of the plane waves $e^{i k \boldsymbol{\theta}_1\cdot\boldsymbol{x}},\dots,e^{i k \boldsymbol{\theta}_N\cdot\boldsymbol{x}}$, the far field pattern of the scattered field vanishes in the directions $-\boldsymbol{\theta}_1,\dots,-\boldsymbol{\theta}_N$.
\end{proposition}

\begin{remark}
In dimension two, the assumptions of this proposition can be slightly simplified. Indeed, we can see that  if $\boldsymbol{\theta}_1,\dots,\boldsymbol{\theta}_N$ denote $N$ vectors of $\mathbb{S}^{1}$, then the elements of the family $\{\boldsymbol{\theta}_{n'}+\boldsymbol{\theta}_{m'}\}_{1\le m'\le n'\le N}$ are all different as soon as they are all non null. 
\end{remark}

\begin{remark}
We have focused our attention only on the case where the propagation of acoustic waves is governed by the equation $\Delta u + k^2 \rho\,u=0$. Let us mention that we can play similarly with the coefficients $A$ or $(A,\rho)$ to obtain the same kind of results when the acoustic field verifies $\div(A\nabla u) + k^2\,u=0$ or $\div(A\nabla u) + k^2\rho\,u=0$.
\end{remark}

\subsection{The case of the scattering direction coinciding with the incident direction}
In the previous constructions (see e.g. (\ref{RelOrtho})), we needed to assume that there holds  $\boldsymbol{\theta}_{\mrm{s}}\neq \boldsymbol{\theta}_{\mrm{i}}$ to control, in the imaginary part of $u_{\mrm{s}}^{\eps\,\infty}(\boldsymbol{\theta}_{\mrm{s}},\boldsymbol{\theta}_{\mrm{i}})$, the terms of orders $\eps^2$, $\eps^3$, $\dots$ by the term of order $\eps$. When $\boldsymbol{\theta}_s=\boldsymbol{\theta}_{\mrm{i}}$ (see Figure \ref{IllustScaPb}), the approach we proposed cannot be implemented. To cope with this problem, one could consider a complex valued parameter $\rho$ or an anisotropic material characterized by some matrix valued coefficient $A$. But imagine that we want to use non dissipative isotropic materials with $A=\mrm{Id}$ only. For a given wavenumber $k>0$, can we find a real valued parameter $\rho$ for the inclusion $\mathcal{D}$ such that the far field pattern associated with the incident plane wave $u_{\mrm{i}}:=e^{i k \boldsymbol{\theta}_{\mrm{i}}\cdot\boldsymbol{x}}$, vanishes in the direction $\boldsymbol{\theta}_{\mrm{i}}$? In the sequel, we shall assume there is only one scattering direction ($N=1$). 

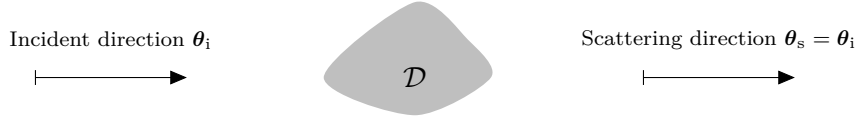
\begin{figure}[!ht]
\centering
\begin{tikzpicture}[scale=1]
\fill [black,line width=1pt,gray!50!white,draw=none] plot [smooth cycle] coordinates {(1,0)(0.7,0.5) (0,1) (-1.2,0) (0,-0.5)};
\node(a)at(0,0){$\mathcal{D}$};
\node(a)at(-4,0.5){{\scriptsize Incident direction $\boldsymbol{\theta}_{\mrm{i}}$}};
\node(a)at(4,0.5){{\scriptsize Scattering direction $\boldsymbol{\theta}_{\mrm{s}}=\boldsymbol{\theta}_{\mrm{i}}$}};
\draw[|-triangle 45] (-5,0) -- (-3,0);
\draw[|-triangle 45] (3,0) -- (5,0);
\end{tikzpicture}
\caption{Scattering problem with $\boldsymbol{\theta}_{\mrm{s}}=\boldsymbol{\theta}_{\mrm{i}}$.\label{IllustScaPb}}
\end{figure}

\noindent According to (\ref{FFPtermeSource}), we know that the far field pattern in the direction $\boldsymbol{\theta}_{\mrm{s}}$ of the scattered field solution of the problem
\begin{equation}\label{pbCla}
\begin{array}{|rcll}
-\Delta u & = & k^2 \rho\,u & \mbox{ in }\R^d,\\
u & = & u_{\mrm{i}}+u_{\mrm{s}} & \mbox{ in }\R^d,\\[4pt]
\multicolumn{4}{|c}{\dsp\lim_{r\to +\infty} r^{\frac{d-1}{2}}\left( \frac{\partial u_{\mrm{s}}}{\partial r}-ik u_{\mrm{s}} \right)  = 0,}
\end{array}
\end{equation}
is given by 
\begin{equation}\label{FormulaFFP}
u_{\mrm{s}}^{\infty}(\boldsymbol{\theta}_{\mrm{s}},\boldsymbol{\theta}_{\mrm{i}}) = c_d\,k^2\dsp \int_{\mathcal{D}} (\rho-1)\,u\,e^{-ik \boldsymbol{\theta}_{\mrm{s}}\cdot\bfx}\,d\bfx.
\end{equation}
When $\boldsymbol{\theta}_{\mrm{s}}=\boldsymbol{\theta}_{\mrm{i}}$, we find, using the relation $-\Delta u_{\mrm{s}}-k^2 \rho\,u_{\mrm{s}}= k^2\,(\rho-1)\,u_{\mrm{i}}$,
\begin{equation}\label{eqIntermFF}
\begin{array}{lcl}
u_{\mrm{s}}^{\infty}(\boldsymbol{\theta}_{\mrm{i}},\boldsymbol{\theta}_{\mrm{i}}) & = & c_d\,k^2\dsp \int_{\mathcal{D}} (\rho-1)\,(u_{\mrm{i}}+u_{\mrm{s}})\,\overline{u_{\mrm{i}}}\,d\bfx\\[10pt]
& = & c_d\,k^2\dsp \int_{\mathcal{D}} (\rho-1)\,|u_{\mrm{i}}|^2\,d\bfx-c_d\,\dsp \int_{\mathcal{D}} u_{\mrm{s}}\,\overline{\Delta u_{\mrm{s}}+k^2\rho\,u_{\mrm{s}}}\,d\bfx\\[10pt]
& = & c_d\,k^2\dsp \int_{\mathcal{D}} (\rho-1)\,|u_{\mrm{i}}|^2\,d\bfx+c_d\,\dsp \int_{\mathcal{D}} |\nabla u_{\mrm{s}}|^2-k^2\rho |u_{\mrm{s}}|^2\,d\bfx+c_d\,
\langle\overline{\partial_{\bfnu}u_{\mrm{s}}},u_{\mrm{s}}\rangle_{\Gamma}.
\end{array}
\end{equation}
Let us work on the last term of the right hand side of the previous equation as in the proof of \cite[Theorem 2.5]{CoKr13}. We take $R$ large enough such that $\overline{\mathcal{D}}$ is contained in the ball $\mrm{B}_R:=\{\bfx\in\R^d\,\big|\,|\bfx| < R\}$ and apply Green formula in $\mrm{B}_R\setminus\overline{\mathcal{D}}$ to obtain
\[
\langle\overline{\partial_{\bfnu}u_{\mrm{s}}},u_{\mrm{s}}\rangle_{\Gamma} = -\langle\overline{\partial_{\bfnu}u_{\mrm{s}}},u_{\mrm{s}}\rangle_{\partial\mrm{B}_R}+\dsp \int_{\mrm{B}_R\setminus\overline{\mathcal{D}}} |\nabla u_{\mrm{s}}|^2-k^2 |u_{\mrm{s}}|^2\,d\bfx.
\]
On $\partial\mrm{B}_R$, $\bfnu$ refers to the unit outward normal to $\mrm{B}_R$. But $u_{\mrm{s}}$ satisfy the radiation condition of (\ref{pbCla}) and admits the decomposition (\ref{scatteredFieldFreeSpaceIntro}). Therefore, there holds 
\begin{equation}\label{eqFF}
\Im m\,\langle\overline{\partial_{\bfnu}u_{\mrm{s}}},u_{\mrm{s}}\rangle_{\Gamma} = k \int_{\mathbb{S}^{d-1}} |u_{\mrm{s}}^{\infty}(\boldsymbol{\theta},\boldsymbol{\theta}_{\mrm{i}})|^2\,d\boldsymbol{\theta}.
\end{equation}
Using (\ref{eqFF}) in (\ref{eqIntermFF}), we conclude
\[
\Im m\,(c_d^{-1}\,u_{\mrm{s}}^{\infty}(\boldsymbol{\theta}_{\mrm{i}},\boldsymbol{\theta}_{\mrm{i}})) = k \int_{\mathbb{S}^{d-1}} |u_{\mrm{s}}^{\infty}(\boldsymbol{\theta},\boldsymbol{\theta}_{\mrm{i}})|^2\,d\boldsymbol{\theta}.
\]
According to the Rellich's lemma (see, e.g., \cite[Theorem 4.1]{CaCo06} in dimension 2 and \cite[Theorem 2.14]{CoKr13} in dimension 3), this proves the
\begin{proposition}\label{PropoFF}
Let $u_{\mrm{s}}$ refer to the scattered field, defined by Problem (\ref{pbCla}), associated with the incident plane wave $u_{\mrm{i}}=e^{i k \boldsymbol{\theta}_{\mrm{i}}\cdot\boldsymbol{x}}$. Then the far field pattern of $u_{\mrm{s}}$ vanishes in the direction $\boldsymbol{\theta}_{\mrm{i}}$ if and only if there holds $u_{\mrm{s}}=0$ in $\R^d\setminus\overline{\mathcal{D}}$.
\end{proposition}
\noindent This shows that it is much more complicated to impose far field invisibility in the direction $\boldsymbol{\theta}_{\mrm{i}}$ than in other directions. In particular, if $k>0$ and $\rho$ are such that the the far field pattern of $u_{\mrm{s}}$ vanishes in the direction $\boldsymbol{\theta}_{\mrm{i}}$, according to Proposition \ref{PropoFF}, we must have $u_{\mrm{s}}=u-u_{\mrm{i}}=0$ and $\partial_{\boldsymbol{\nu}}u_{\mrm{s}}=\partial_{\boldsymbol{\nu}}(u-u_{\mrm{i}})=0$ on $\Gamma$. This implies that the pair $(u,u_{\mrm{i}})=(u_{\mrm{i}}+u_{\mrm{s}},u_{\mrm{i}})$ verify
\begin{equation}\label{pb trans restrictif}
\begin{array}{|lcll}
\Delta u +k^2 \rho\,u &=&0&\mbox{ in }\mathcal{D},\\
\Delta u_{\mrm{i}}+k^2 u_{\mrm{i}}&=&0&\mbox{ in }\mathcal{D},\\
u-u_{\mrm{i}} &=&0&\mbox{ on }\Gamma,\\
\partial_{\boldsymbol{\nu}}( u-u_{\mrm{i}})&=&0&\mbox{ on }\Gamma.
\end{array}
\end{equation}
This is nothing else than a transmission eigenvalue problem with a strong hypothesis on $u_{\mrm{i}}$. Compared to the classical transmission eigenvalue problem where the restriction of the incident field to $\mathcal{D}$ belongs to $\mrm{L}^2(\mathcal{D})$, here we additionally impose $u_{\mrm{i}}=e^{i k \boldsymbol{\theta}_{\mrm{i}}\cdot\boldsymbol{x}}$. As a consequence, for a given $\rho$ which verifies $0<c\le \rho \le C<1$ in a neighbourhood of $\Gamma$ or $1<c\le \rho$ in a neighbourhood of $\Gamma$, we know according to \cite{Sylv12} (see also \cite{LaVa13,Robb13} in the case of a smooth $\rho$) that the set of wavenumbers for which (\ref{pb trans restrictif}) has a non trivial solution is either empty or discrete. However, we cannot use the results of existence of transmission eigenvalues (see \cite{PaSy08,CaGH10}) to conclude to the existence of non trivial solutions to (\ref{pb trans restrictif}) because here, we impose a very restrictive condition for $u_{\mrm{i}}$. More interesting for our configuration are the recent results of \cite{BlPS14,PaSV14,ElHu15}. 
In these articles, situations where the support of $\rho-1$ has corners or edges are considered. In this context, the authors provide sufficient criteria on $\rho$ which ensure that all non trivial incident Herglotz functions (\textit{i.e.} all non trivial generalized combinations of plane waves), for all wavenumbers $k>0$, produce non trivial scattered field. For such coefficients $\rho$, the only solution of $(\ref{pb trans restrictif})$ is null. For more general $\rho$, for example such that the support of $\rho-1$ is smooth, even assuming $u_{\mrm{i}}=e^{i k \boldsymbol{\theta}_{\mrm{i}}\cdot\boldsymbol{x}}$, we do not know any result showing that $(\ref{pb trans restrictif})$ has only the trivial solution. Nevertheless, in any case, and this is the important message of this discussion, it seems particularly delicate to impose far field invisibility in the direction $\boldsymbol{\theta}_{\mrm{i}}$.

\subsection{Numerical experiments}\label{SubSectionNumExpe}
We want to implement numerically the approach developed in \S\ref{paragraphInvisi}. For a given wavenumber $k>0$, we consider the 2D scattering problem 
\begin{equation}\label{PbChampTotalFreeSpaceIntrodNum}
\begin{array}{|rcll}
\multicolumn{4}{|l}{\mbox{Find }u\in\mH^1_{\loc}(\R^2)\mbox{ such that}}\\[4pt]
-\Delta u & = & k^2 \rho\,u & \mbox{ in }\R^2,\\
u & = & u_{\mrm{i}}+u_{\mrm{s}} & \mbox{ in }\R^2,\\[4pt]
\multicolumn{4}{|c}{\dsp\lim_{r\to +\infty} \sqrt{r}\left( \frac{\partial u_{\mrm{s}}}{\partial r}-ik u_{\mrm{s}} \right)  = 0,}
\end{array}
\end{equation}
with $u_{\mrm{i}}=e^{i k \boldsymbol{\theta}_{\mrm{i}}\cdot\boldsymbol{x}}$ (one incident direction $\boldsymbol{\theta}_{\mrm{i}}=(\cos(\psi_{\mrm{i}}),\,\sin(\psi_{\mrm{i}}))$). Our goal is to build $\rho$ such that the far field pattern of $u_{\mrm{s}}$ vanishes in the three scattering directions $\boldsymbol{\theta}_{1}=(\cos(\psi_{1}),\,\sin(\psi_{1}))$, $\boldsymbol{\theta}_{2}=(\cos(\psi_{2}),\,\sin(\psi_{2}))$ and $\boldsymbol{\theta}_{3}=(\cos(\psi_{3}),\,\sin(\psi_{3}))$. In other words, we take $N=3$. Consider $\mathcal{D}$ a given Lipschitz domain. Following (\ref{ExpressMu}), we search for $\rho$ of the form $\rho=1+\eps\mu$, where $\mu$ is a function supported in $\overline{\mathcal{D}}$ such that
\begin{equation}\label{vecTau}
\mu=\mu_0+\dsp\sum_{n=1}^3 \tau_{1,n}\,\mu_{1,n} + \dsp\sum_{n=1}^3 \tau_{2,n}\, \mu_{2,n}.
\end{equation}
To define functions $\mu_{1,n}$, $\mu_{2,n}$ that satisfy the conditions of (\ref{RelOrtho}), we start by computing the matrix
\[
\mathbb{B} = \left(\begin{array}{cc}
\mathbb{B}^{11} & \mathbb{B}^{12}\\[10pt]
\mathbb{B}^{21} & \mathbb{B}^{22}\end{array}\right)
\]
where $\mathbb{B}^{21}=(\mathbb{B}^{12})^{\top}$ and 
\[
\begin{array}{|lclcl}
\mathbb{B}^{11} & = & (\mathbb{B}^{11}_{n,n'})_{3\times3} & \mbox{ with } & \mathbb{B}^{11}_{n,n'} =\dsp\int_{\mathcal{D}}\cos(k(\boldsymbol{\theta}_{\mrm{i}}-\boldsymbol{\theta}_{n})\cdot\boldsymbol{x})\,\cos(k(\boldsymbol{\theta}_{\mrm{i}}-\boldsymbol{\theta}_{n'})\cdot\boldsymbol{x})\,\,d\boldsymbol{x}\\[15pt]
\mathbb{B}^{12} & = & (\mathbb{B}^{12}_{n,n'})_{3\times3} & \mbox{ with } & \mathbb{B}^{12}_{n,n'} =\dsp\int_{\mathcal{D}}\sin(k(\boldsymbol{\theta}_{\mrm{i}}-\boldsymbol{\theta}_{n})\cdot\boldsymbol{x})\,\cos(k(\boldsymbol{\theta}_{\mrm{i}}-\boldsymbol{\theta}_{n'})\cdot\boldsymbol{x})\,\,d\boldsymbol{x}\\[15pt]
\mathbb{B}^{22} & = & (\mathbb{B}^{22}_{n,n'})_{3\times3} & \mbox{ with } & \mathbb{B}^{22}_{n,n'} =\dsp\int_{\mathcal{D}}\sin(k(\boldsymbol{\theta}_{\mrm{i}}-\boldsymbol{\theta}_{n})\cdot\boldsymbol{x})\,\sin(k(\boldsymbol{\theta}_{\mrm{i}}-\boldsymbol{\theta}_{n'})\cdot\boldsymbol{x})\,\,d\boldsymbol{x}.
\end{array}
\]
Observing that $\{\cos(k(\boldsymbol{\theta}_{\mrm{i}}-\boldsymbol{\theta}_{1})\cdot\boldsymbol{x})),\dots,\cos(k(\boldsymbol{\theta}_{\mrm{i}}-\boldsymbol{\theta}_{3})\cdot\boldsymbol{x})),\sin(k(\boldsymbol{\theta}_{\mrm{i}}-\boldsymbol{\theta}_{1})\cdot\boldsymbol{x})),\dots,\sin(k(\boldsymbol{\theta}_{\mrm{i}}-\boldsymbol{\theta}_{3})\cdot\boldsymbol{x}))\}$ is a family of linearly independent functions on $\mathcal{D}$ (when there holds $\boldsymbol{\theta}_{\mrm{i}}\ne\boldsymbol{\theta}_{n}$ for $n=1\dots 3$), we can prove that $\mathbb{B}$ is invertible. We denote $\mathbb{D}$ its inverse. Finally, we take 
\[
\begin{array}{|lcl}
\mu_{1,n} & = & \dsp\sum_{n'=1}^3 \mathbb{D}_{n',n}\,\cos(k(\boldsymbol{\theta}_{\mrm{i}}-\boldsymbol{\theta}_{n'})\cdot\boldsymbol{x})+\dsp\sum_{n'=1}^3 \mathbb{D}_{3+n',n}\,\sin(k(\boldsymbol{\theta}_{\mrm{i}}-\boldsymbol{\theta}_{n'})\cdot\boldsymbol{x})\\[10pt]
\mu_{2,n} & = & \dsp\sum_{n'=1}^3 \mathbb{D}_{n',3+n}\,\cos(k(\boldsymbol{\theta}_{\mrm{i}}-\boldsymbol{\theta}_{n'})\cdot\boldsymbol{x})+\dsp\sum_{n'=1}^3 \mathbb{D}_{3+n',3+n}\,\sin(k(\boldsymbol{\theta}_{\mrm{i}}-\boldsymbol{\theta}_{n'})\cdot\boldsymbol{x}).
\end{array}
\]
Then, we construct a $\mu_0$ that verifies the six orthogonality conditions of (\ref{RelOrtho}) taking 
\begin{equation}\label{defMu0}
\mu_0 = \mu^{\#}_0 - \dsp\sum_{n=1}^3 \left(\dsp\int_{\mathcal{D}}\mu_{1,n}\,\mu^{\#}_0\,d\boldsymbol{x}\right)\,\mu_{1,n} - \dsp\sum_{n=1}^3 \left(\dsp\int_{\mathcal{D}}\mu_{2,n}\,\mu^{\#}_0\,d\boldsymbol{x}\right)\, \mu_{2,n}
\end{equation}
where $\mu^{\#}_0$ is an arbitrary function such that $\mu^{\#}_0\notin\mrm{span}\{\mu_{1,1},\dots,\mu_{1,3},\mu_{2,1},\dots,\mu_{2,3}\}$. Here, we choose $\mu^{\#}_0$ such that $\mu^{\#}_0(\boldsymbol{x})=1+x+y$.\\
\newline
For the numerical experiments, we take $\psi_{\mrm{i}}=0^{\circ}$, $\psi_{1}=90^{\circ}$, $\psi_{2}=180^{\circ}$ (backscattering direction), $\psi_{3}=225^{\circ}$ and $\mathcal{D}=\mrm{B}_1$ (so that the inclusion in contained in the unit disk). The wavenumber $k$ is set to $k=4$. Now, we describe the procedure to solve the fixed point problem (\ref{PbFixedPoint}) by induction.\\
\newline
We denote $\tau^{j}=(\tau^{j}_{1,1},\dots,\tau^{j}_{1,3},\tau^{j}_{2,1},\dots,\tau^{j}_{2,3})^{\top}$ (resp. $\mu^j$) the value of $\tau=(\tau_{1,1},\dots,\tau_{1,3},\tau_{2,1},\dots,\tau_{2,3})^{\top}$ (resp. $\mu$) at iteration $j\ge0$ (we remind the reader that $\tau_{1,n}$, $\tau_{2,n}$ are the parameters appearing in (\ref{vecTau})). Using formulas (\ref{FFPTotal}), (\ref{PbFixedPoint}), for $j\ge0$, $n=1,\dots,3$, we define
\begin{equation}\label{EcriturePtFixe}
\begin{array}{|lcl}
\tau^{j+1}_{1,n} = \tau^{j}_{1,n} - \Re e\,(\,(\eps\,c_d\,k^2)^{-1}\,u_{\mrm{s}}^{\eps\,j\,\infty}(\boldsymbol{\theta}_n,\boldsymbol{\theta}_{\mrm{i}})\,)\\[12pt]
\tau^{j+1}_{2,n} = \tau^{j}_{2,n} - \Im m\,(\,(\eps\,c_d\,k^2)^{-1}\,u_{\mrm{s}}^{\eps\,j\,\infty}(\boldsymbol{\theta}_n,\boldsymbol{\theta}_{\mrm{i}})\,).
\end{array}
\end{equation}
In the above definition, $u_{\mrm{s}}^{\eps\,j\,\infty}(\boldsymbol{\theta}_n,\boldsymbol{\theta}_{\mrm{i}})$ denotes the far field pattern in the direction $\boldsymbol{\theta}_n$ of the function $u_{\mrm{s}}^{\eps\,j}$ satisfying the problem 
\begin{equation}\label{PbChampScaNum}
\begin{array}{|rcll}
\multicolumn{4}{|l}{\mbox{Find }u_{\mrm{s}}^{\eps\,j}\in\mH^1_{\loc}(\R^2)\mbox{ such that}}\\[4pt]
-\Delta u_{\mrm{s}}^{\eps\,j}-k^2 (1+\eps\,\mu^j)\,u_{\mrm{s}}^{\eps\,j} & = & k^2\,\eps\,\mu^j\,u_{\mrm{i}}& \mbox{ in }\R^2,\\[4pt]
\multicolumn{4}{|c}{\dsp\lim_{r\to +\infty} \sqrt{r}\left( \frac{\partial u_{\mrm{s}}^{\eps\,j}}{\partial r}-ik u_{\mrm{s}}^{\eps\,j} \right)  = 0.}
\end{array}
\end{equation}
According to formula (\ref{FormulaFFP}), we know that 
\begin{equation}\label{DefTermeChampLointain}
u_{\mrm{s}}^{\eps\,j\,\infty}(\boldsymbol{\theta}_n,\boldsymbol{\theta}_{\mrm{i}}) = c_d\,k^2\dsp \int_{\mathcal{D}} \eps\,\mu^j\,(u_{\mrm{i}}+u_{\mrm{s}}^{\eps\,j})\,e^{-ik \boldsymbol{\theta}_n}\,d\bfx.
\end{equation}
We approximate the solution of Problem (\ref{PbChampScaNum}) with a P2 finite element method set on the ball $\mrm{B}_8$ ($8$ is the radius). On $\partial\mrm{B}_8$, a truncated Dirichlet-to-Neumann map with 13 harmonics serves as a transparent boundary condition. We choose $\tau^{0}=(0,\dots,0)$. For the simulations of Figures \ref{figResult1}--\ref{figResult4}, we stop the procedure when $\sum_{n=1}^3|\tau^{j+1}_{1,n}- \tau^{j}_{1,n}|+|\tau^{j+1}_{2,n}- \tau^{j}_{2,n}|\le 10^{-13}$ (corresponding to 37 iterations) and we take $\eps=0.15$. To obtain the results of Figures \ref{figResult5}--\ref{ValeursInclus}, we perform 10 iterations and we try several values of $\eps$. For the computations, we use the \textit{FreeFem++}\footnote{\textit{FreeFem++}, \url{http://www.freefem.org/ff++/}.} software while we display the results with \textit{Matlab}\footnote{\textit{Matlab}, \url{http://www.mathworks.se/}.} and \textit{Paraview}\footnote{\textit{Paraview}, \url{http://www.paraview.org/}.}.

\begin{figure}[!ht]
\centering
\includegraphics[scale=0.11]{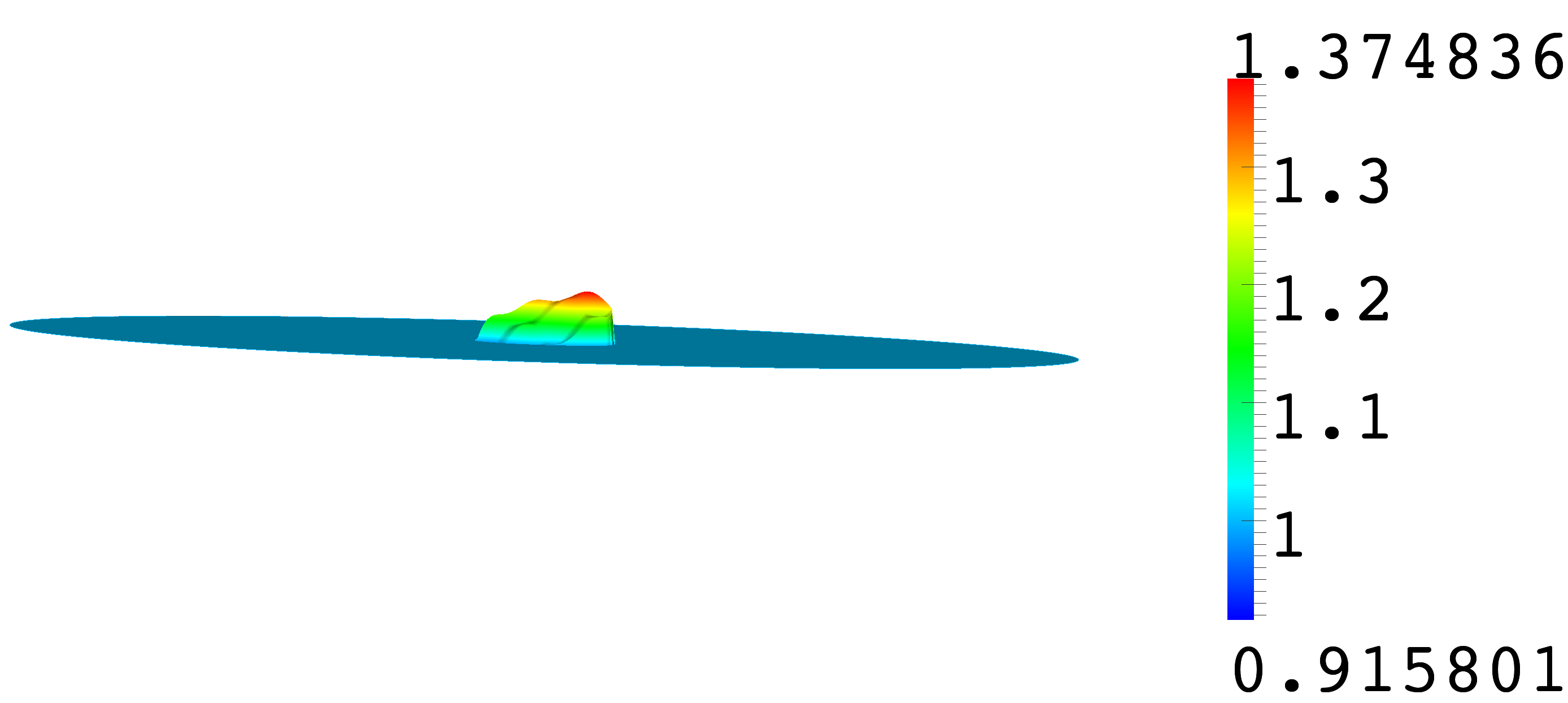}
\caption{Coefficient $\rho$. Interestingly, the fixed point procedure converges though the coefficient $\rho$ is not a very small perturbation of the parameter (uniformly equal to one) of the reference medium. The domain represented here is equal to $\mrm{B}_8$ and the inclusion $\mathcal{D}$ is located in $\mrm{B}_1$. \label{figResult1}}
\end{figure}

\begin{figure}[!ht]
\centering
\includegraphics[scale=0.11]{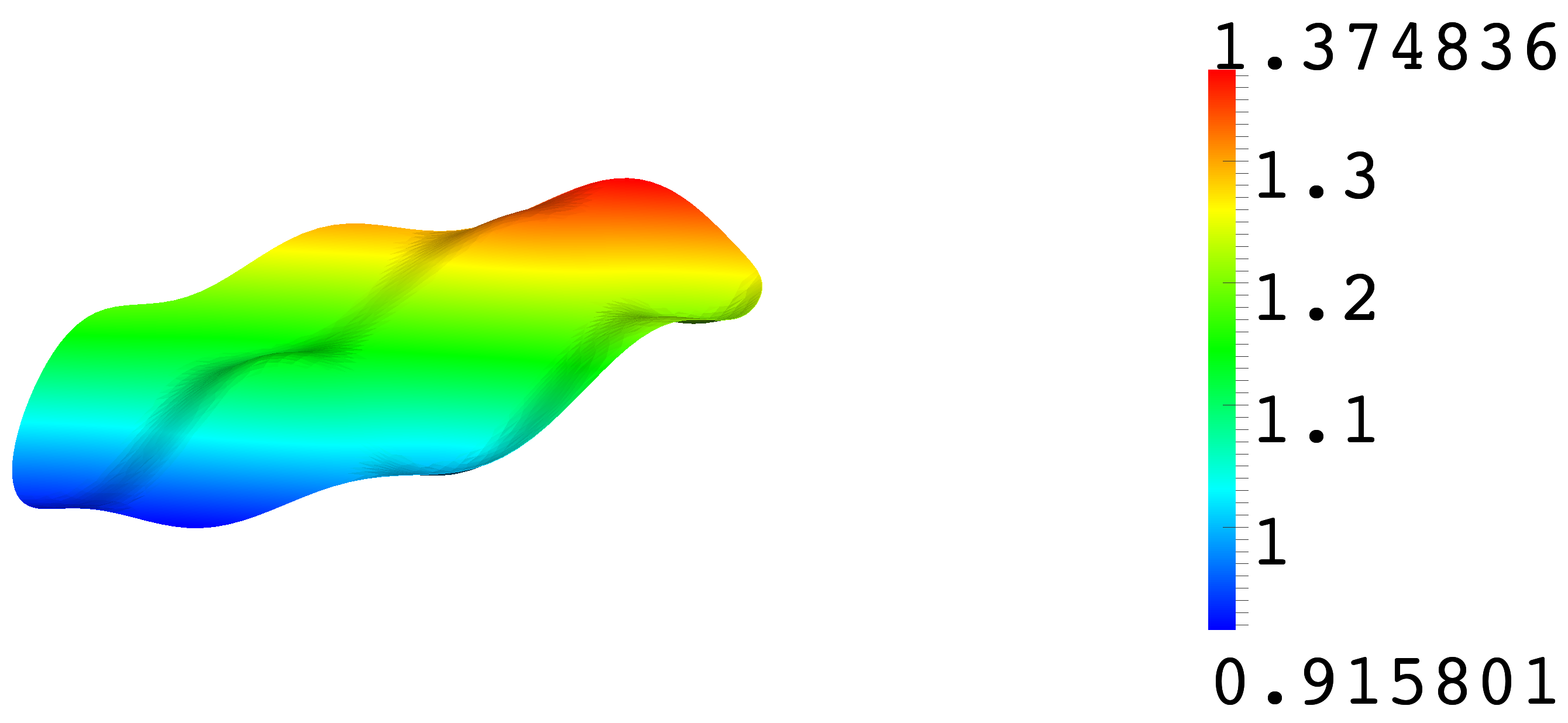}
\caption{Restriction of the coefficient $\rho$ to $\mathcal{D}$. In this particular case, with our choice for the functions defining $\mu$, we have $\rho|_{\overline{\mathcal{D}}}\in\mathscr{C}^{\infty}(\overline{\mathcal{D}})$. The domain represented here is $\mrm{B}_1$. \label{figResult2}
}
\end{figure}

\begin{figure}[!ht]
\centering
\includegraphics[scale=0.11]{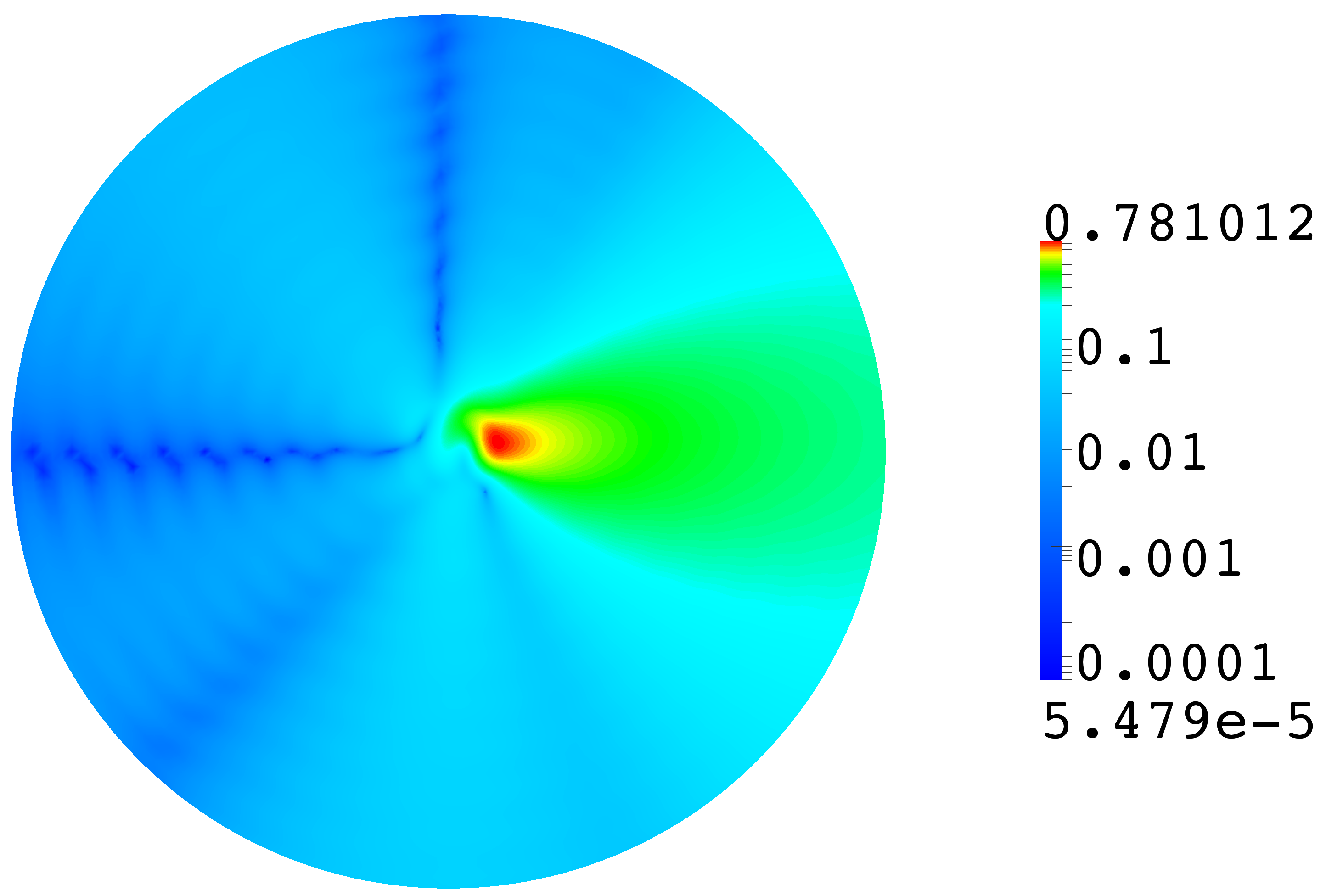}
\caption{Absolute value of the approximation of the scattered field $u_{\mrm{s}}^{\eps}$ at the end of the fixed point procedure in logarithmic scale. As desired, we see it is very small far from $\mathcal{D}$ in the directions corresponding to the angles $90^{\circ}$, $180^{\circ}$ and $225^{\circ}$. The domain represented here is equal to $\mrm{B}_8$ and the inclusion $\mathcal{D}$ is located in $\mrm{B}_1$.\label{figResult3}}
\end{figure}

\begin{figure}[!ht]
\centering
\includegraphics[scale=0.9]{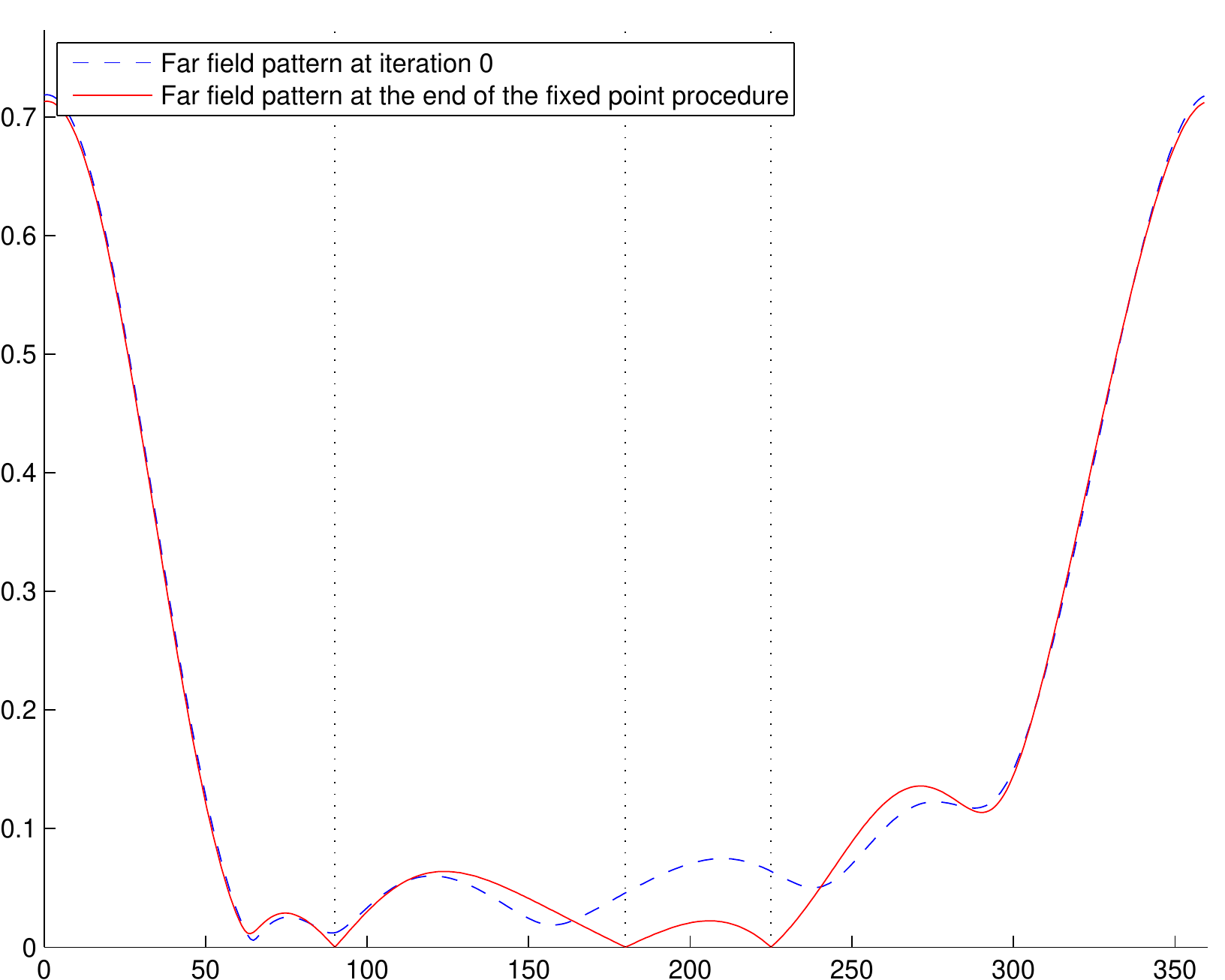}
\caption{The dashed curve represents the approximated far field pattern at iteration 0 (taking $\rho=1+\eps\mu^0=1+\eps\mu_0$). The solid curve corresponds to the approximated far field pattern computed for the parameter $\rho$ obtained at the end of the fixed point procedure. The dotted lines indicate the directions where we want $u_{\mrm{s}}^{\eps\,\infty}$ to vanish. From these results, and in accordance with estimate (\ref{RelAPos}), we infer that the term $\mu_0$ in the expression of $\mu$ (see (\ref{vecTau})) determines the shape of $u_{\mrm{s}}^{\eps\,\infty}$. But, we also notice that the fixed point procedure allows to correct significantly $u_{\mrm{s}}^{\eps\,\infty}$ in the directions $\boldsymbol{\theta}_1$, $\boldsymbol{\theta}_2$ and $\boldsymbol{\theta}_3$.\label{figResult4}}
\end{figure}

\begin{figure}[!ht]

\centering
\includegraphics[scale=0.9]{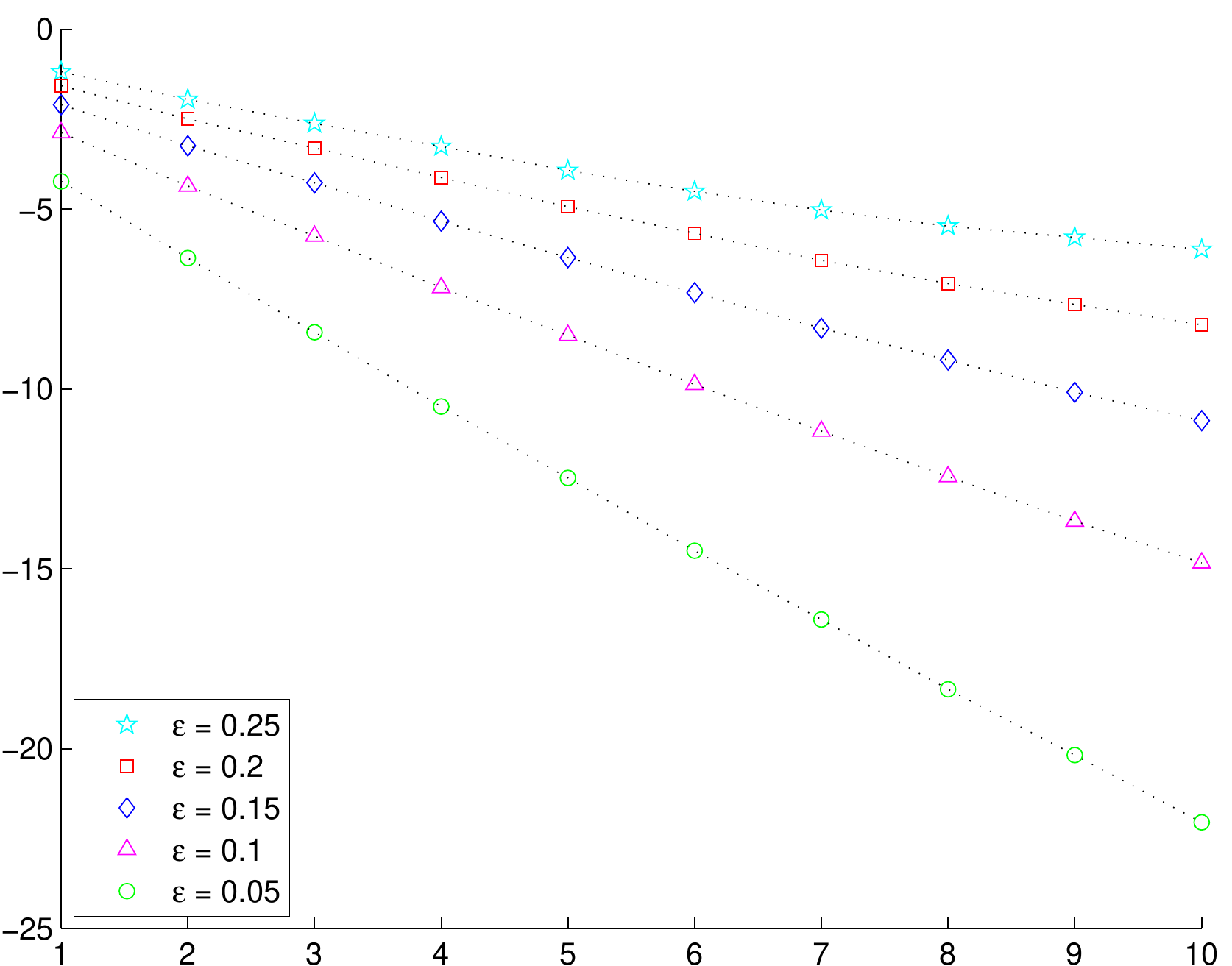}
\caption{Curves $\ln \sum_{n=1}^3 |u_{\mrm{s}}^{\eps\,j\,\infty}(\boldsymbol{\theta}_n,\boldsymbol{\theta}_{\mrm{i}})|$ (see the definition of the term $u_{\mrm{s}}^{\eps\,j\,\infty}(\boldsymbol{\theta}_n,\boldsymbol{\theta}_{\mrm{i}})$ in (\ref{DefTermeChampLointain})) with respect to the number of iterations $j$ for several values of $\eps$. According to (\ref{EcriturePtFixe}) and (\ref{EstimContra}), we know that there holds $\ln \sum_{n=1}^3 |u_{\mrm{s}}^{\eps\,j\,\infty}(\boldsymbol{\theta}_n,\boldsymbol{\theta}_{\mrm{i}})| \le j \ln\eps + C$, where $C$ is a constant independent of $\eps$. The results we obtain are in agreement with this estimate. This shows that the convergence of the fixed point procedure gets quicker as $\eps$ goes to zero. But in this case, the perturbation of the reference medium becomes smaller and smaller, as indicated by the results of Figure \ref{ValeursInclus}.\label{figResult5}}
\end{figure}

\begin{figure}[!ht]
\centering
\renewcommand{\arraystretch}{1.5}
\begin{tabular}{|c|c|c|}
\hline
$\eps$ & $\underset{\mathcal{D}}{\min}$ $\rho$ & $\underset{\mathcal{D}}{\max}$ $\rho$\\
\hline
0.25 & 0.818267 & 1.61006 \\
\hline
0.2 & 0.870935 & 1.49338 \\
\hline
0.15 & 0.915802 & 1.37484 \\
\hline
0.1 & 0.9436 & 1.25374 \\
\hline
0.05 & 0.968809 & 1.1291 \\[2pt]
\hline
\end{tabular}
\caption{ Minima and maxima of the parameter $\rho$ obtained after 10 iterations for several values of $\eps$.  \label{ValeursInclus}}
\end{figure}

\newpage
\clearpage

\section*{Acknowledgments}
The work of the first author is supported by the Academy of Finland (decision 140998) and by the FMJH through the grant ANR-10-CAMP-0151-02 in the ``Programme des Investissements d'Avenir''. The research of the third author is supported by the
Russian Foundation of Fundamental Investigations, grant No. 15-01-02175.

\bibliographystyle{plain}
\bibliography{Bibli}
\end{document}